\documentclass[letterpaper,11pt,reqno]{amsart}

\makeatletter
\usepackage{amssymb}
\usepackage{mathtools}
\usepackage{latexsym}
\usepackage{amsbsy}
\usepackage{bbm}
\usepackage{amsfonts}
\usepackage{graphicx}
\usepackage{caption}
\usepackage{enumerate}
\usepackage{enumitem}
\usepackage{mathrsfs}
\usepackage{dsfont}
\usepackage{url}
\usepackage{color}
\usepackage{algorithm}
\usepackage{algpseudocode}
\usepackage[citecolor=blue, urlcolor=blue, linkcolor=red, colorlinks=true, bookmarksopen=true]{hyperref}
\allowdisplaybreaks
\usepackage{circuitikz}

\mathtoolsset{showonlyrefs}

\mathtoolsset{centercolon}

\def\marginpar#1{\ignorespaces}

\makeatletter
 \def\@textbottom{\vskip \z@ \@plus 9.52pt}
 \let\@texttop\relax
\makeatother

\DeclareMathOperator\PI{PI}
\DeclareMathOperator\LSI{LSI}
\DeclareMathOperator\ent{Ent}
\DeclareMathOperator\var{Var}
\DeclareMathOperator\E{\mathbb{E}}
\DeclareMathOperator\Id{\text{Id}}

\newtheorem{theorem}{Theorem}[section]
\newtheorem{lemma}[theorem]{Lemma}
\newtheorem{proposition}[theorem]{Proposition}
\newtheorem{corollary}[theorem]{Corollary}
\newtheorem{definition}[theorem]{Definition}
\newtheorem{remark}[theorem]{Remark}
\newtheorem{assump}[theorem]{Assumption}

\numberwithin{equation}{section}
\makeatother
\begin{document}
\title[Ergodicity of the infinite swapping algorithm]{Ergodicity of the infinite swapping algorithm at low temperature}

\author[Georg Menz]{{Georg} Menz}
\address{Department of Mathematics, UCLA. 
} \email{menz@math.ucla.edu}

\author[Andr\'e Schlichting]{{Andr\'e} Schlichting}
\address{Institut for Analysis and Numerics, WWU Münster.}
\email{a.schlichting@uni-muenster.de}

\author[Wenpin Tang]{{Wenpin} Tang}
\address{Department of Industrial Engineering and Operations Research, Columbia University.
} \email{wt2319@columbia.edu}

\author[Tianqi Wu]{{Tianqi} Wu}
\address{Department of Mathematics, UCLA.
} \email{timwu@ucla.edu}

\date{\today} 
\begin{abstract}
Sampling Gibbs measures at low temperatures is an important task but computationally challenging. 
Numerical evidence suggests that the infinite-swapping algorithm (isa) is a promising method. 
The isa can be seen as an improvement of the replica methods. 
We rigorously analyze the ergodic properties of the isa in the low temperature regime, deducing an Eyring-Kramers formula for the spectral gap (or Poincar\'e constant) and an estimate for the log-Sobolev constant. 
Our main results indicate that the effective energy barrier can be reduced drastically using the isa compared to the classical over-damped Langevin dynamics. 
As a corollary, we derive a deviation inequality showing that sampling is also improved by an exponential factor. 
Finally, we study simulated annealing for the isa and prove that the isa outperforms again the over-damped Langevin dynamics.
\end{abstract}

\maketitle
\smallskip


\section{Introduction}

Sampling from Gibbs measures at low temperatures is important in science and engineering.
It has a variety of applications including molecular dynamics \cite{Andersen80, CS11} and Bayesian inference \cite{RC05, GS13}.
Usually, sampling at low temperatures is slow due to the fact that at low temperatures energy barriers in the underlying energy landscape are large. 
This traps the stochastic sampling process and slows down sampling.

One popular way to sample Gibbs measures is to run the over-damped Langevin equation or its various discretization schemes for approximation, see e.g. \cite{RT96, Dala17, DM17, DC19}.
A lot of efforts have been made to accelerate sampling at low temperatures and there are many competing methods.
One of them is the replica exchange method which is also known as parallel tempering. 
In the simplest version of a replica exchange method, one considers two particles governed by independent copies of the underlying dynamics, for instance, the over-damped Langevin equation.
One particle evolves at the desired low temperature $\tau_1>0$, and the other particle evolves at a higher temperature $\tau_2 > 0$ with $ \tau_1 \ll \tau_2 \ll 1$.
At some random times, the positions of both particles are swapped. 
This approach has the advantage that the particle at a low temperature correctly samples the low-temperature Gibbs measure
whereas the particle at a high temperature can explore the full state space, and discover the relevant states of the system efficiently.

Replica exchange methods or parallel tempering have been successfully applied in many different scenarios, and they seem to accelerate sampling in low-temperature situations quite well.
As far as we are concerned, almost all evaluations of the performance of those methods are empirical. 
In an attempt to study the sampling performance of parallel tempering, it was discovered in \cite{Dupuis12} that the large deviation rate function for time-averaged empirical measures of parallel tempering is a monotone function of the swapping rate. 
It implies that sampling only improves at a faster swapping rate. 

This led to the question of a suitable limiting process as the swapping rate goes to infinity. Since the number of jumps of the particles would grow to infinity in any bounded time-interval, the authors in~\cite{Dupuis12} suggest the infinite swapping algorithm/process (isa), a procedure that can be interpreted as the limit of parallel tempering, where instead of the \emph{particle positions}, the \emph{particle temperatures} are swapped at an infinite fast rate (see Section \ref{s_isa_as_limit_of_pt} for a review).

To be more precise, let $H: \mathbb{R}^n \to \mathbb{R}$ be the underlying energy landscape and the goal is to sample the Gibbs measure with density $\nu^{\tau_1}(x):= \frac{1}{Z^{\tau_1}} \exp\left(-\frac{H(x)}{\tau_1} \right)$ where $Z^{\tau_1}$ is the normalizing constant.
Formally, given two different temperatures $0<\tau_1 \ll \tau_2$, the isa is defined as the evolution of two particles $X_1= (X_1(t), \, t \ge 0)$ and $X_2= (X_2(t), \, t \ge 0)$ governed by the stochastic differential equations (SDEs):
\begin{equation}
  \label{e_infinite_swapping}
  \left\{ \begin{array}{rcl}
dX_1 = - \nabla  H(X_1)\, dt + \sqrt{2 \tau_1 \rho(X_1 , X_2) + 2 \tau_2 \rho(X_2, X_1) } \, d B_1 ,\\ 
dX_2 = - \nabla  H(X_2)\, dt + \sqrt{2 \tau_2 \rho(X_1 , X_2) + 2 \tau_1 \rho(X_2, X_1) } \, d B_2 ,
\end{array}\right.
 \end{equation}
where $(B_1, B_2)$ are independent Brownian motions in $\mathbb{R}^n$, and 
\begin{equation}\label{e:def:rho:pi}
  \rho (x_1,x_2) :=\frac{\pi (x_1, x_2)}{\pi (x_1, x_2) + \pi (x_2, x_1)} \quad \text{and} \quad \pi (x_1, x_2) := \nu^{\tau_1}(x_1) \nu^{\tau_2}(x_2).
\end{equation}
Since $\tau_1 \ne \tau_2$, we have that $\pi(x_1,x_2) \ne \pi(x_2, x_1)$, and thus $\rho (x_1,x_2)  \ne  \rho (x_2,x_1)$. The functions $\rho(x_1,x_2), \rho(x_2,x_1)$ are relative weights assigned to the two configurations $(x_1,x_2)$, $(x_2,x_1)$ based on $\pi$. At each moment, this essentially assigns the higher temperature $\tau_2$ to the particle whose potential energy $H$ is higher at that moment (see also~\cite[Section 3.2]{Doll17}). 

The crucial feature of the dynamics~\eqref{e_infinite_swapping} is that the empirical measure
\[
\eta_t: = \frac{1}{t} \int_0^t \rho(X_1, X_2) \delta_{(X_1, X_2)} + \rho(X_2, X_1) \delta_{(X_2, X_1)} ds 
\]
converges weakly to the product measure $\pi$ as $t \to \infty$ by the ergodic theorem. 
In particular, by restricting to the first coordinate, the measure $\frac{1}{t} \int_0^t \rho(X_1, X_2) \delta_{X_1} + \rho(X_2, X_1) \delta_{X_2} ds$ approximates the Gibbs measure $\nu^{\tau_1}$ for $t$ large enough.
In \cite{Dupuis12}, a large deviation principle was established for the measure $\eta_t$. 
However, it is not clear how the rate function depends on the temperatures $(\tau_1,\tau_2)$, so it is less obvious why the higher temperature $\tau_2$ may be helpful.
Further numerical and heuristic studies in \cite{Doll17} indicate that there is an exponential gain when using the isa for sampling in comparison with the classical over-damped Langevin dynamics.
Recently the isa was applied to training restricted Boltzmann machines \cite{HNR20}, and was shown to be competitive empirically.
But no rigorous result has been established so far on how well the isa accelerates sampling at low temperatures.

In this article we take the analysis of \cite{Dupuis12, Doll17} to the next level through a functional inequality approach.
We carry out the first rigorous study of the ergodic properties of the isa at low temperatures by quantifying its convergence in terms of the temperatures ($\tau_1, \tau_2)$.
Under standard nondegeneracy assumptions,
we deduce the low-temperature asymptotics for the Poincar\'e and the log-Sobolev constant of the isa, see Theorem \ref{thm:PI} and Theorem \ref{thm:LSI} below. 
In the context of metastability, these formulas are also known as Eyring-Kramers formulas (see \cite{Ber13} for background).
Comparing our results to the Eyring-Kramers formulas for the over-damped Langevin equation (e.g. see \cite{Bovier04, Bovier05, MS14}), 
we have an exponential gain: the effective energy barrier of the underlying energy landscape~$H$ only sees the higher temperature $\tau_2$. 
We also give indications that our results are optimal.

To the best of our knowledge, this is the first time an Eyring-Kramers formula was derived for inhomogeneous diffusions, 
for which the stationary and ergodic distribution is generally unknown.
By construction, however, the isa \eqref{e_infinite_swapping} has an explicit stationary distribution $\mu$ given by $\mu(x_1, x_2) = \frac{1}{2} \left(\pi (x_1,x_2) + \pi (x_2, x_1) \right)$, where $\pi(\cdot, \cdot)$ is defined by \eqref{e:def:rho:pi}.
This makes a rigorous analysis of \eqref{e_infinite_swapping} feasible. For the proof of our main results, Theorem \ref{thm:PI} and Theorem \ref{thm:LSI}, we follow the transportation approach of \cite{MS14}.
The idea is to identify the right ``paths'' of transport which give the leading order term in the Poincar\'e and the log-Sobolev constant of the isa. In the case of the Langevin diffusion process those paths can be obtained from mountain pass paths between local minima of the energy $H$.  
Since the isa is a process on $\mathbb{R}^n \times \mathbb{R}^n$ swapping the two particle temperatures, it requires analyzing transport in a planar network obtained from the product structure of two energies, and so is more involved.

There are several other methods which could be used to deduce the Eyring-Kramers formula for the Poincar\'e constant. 
For instance, one could consider adapting the potential theoretic approach (see \cite{Bovier04, Bovier05}), or the semiclassical analysis (see \cite{HeKlNi04, HeNi05, HeNi06}), or the approach using quasi-stationarity (see \cite{BR16, LL19}).
We adopt the approach of \cite{MS14}, which is robust enough to deduce the Eyring-Kramers formula for the log-Sobolev constant in the setting of an inhomogeneous diffusion coefficient. 
The rate of convergence in relative entropy obtained from the log-Sobolev constant is important for our applications to sampling and simulated annealing.

In the first application, we apply the main results to study the sampling properties of the isa and compare it to the over-damped Langevin dynamics. 
It is well known that the Poincar\'e and the log-Sobolev constants characterize the rate of convergence to equilibrium of the underlying process. 
It is also known that Poincar\'e and log-Sobolev inequalities yield non-asymptotic concentration/deviation inequalities (see \cite{CaGu08,WuYa08} and references therein). 
Hence, our main results yield a quantitative control in terms of the temperatures $(\tau_1, \tau_2)$ on the rate of convergence of the time average to the ensemble average, quantifying the ergodic theorem.
Let us note in comparison that the precise dependence on $(\tau_1, \tau_2)$ is missing in the large deviation estimates for the isa in \cite{Dupuis12}.
As a byproduct of our analysis, we find a condition on $(\tau_1, \tau_2)$ under which sampling at low temperatures using the isa is exponentially faster than using the over-damped Langevin dynamics. 
This provides a guidance on the choice of the higher temperature $\tau_2$ for the isa.

In the second application, we study the isa for simulated annealing and compare it to simulated annealing adapted to the over-damped Langevin dynamics. 
Simulated annealing (SA) is an umbrella term denoting a particular set of stochastic optimization methods. 
SA can be used to find the global extremum of a function $H:\mathbb{R}^n \to \mathbb{R}$, in particular when $H$ is non-convex. 
Those methods have many applications in different fields, for example in physics, chemistry and operations research (see e.g. \cite{Laar87, KoAnja94, Nar99}). 
The name and inspiration comes from annealing in metallurgy, a process that aims to increase the size of the crystals by heating and controlled cooling. 
The SA mimics this procedure mathematically. 
The stochastic version of SA was independently described by Kirkpatrick, Gelatt and Vecchi \cite{KGV} and {\v{C}}ern\'y \cite{Cer}.
See Section \ref{s4} for details on simulated annealing.

Replica exchange methods or parallel tempering have been successfully applied to nonconvex optimization (see e.g. \cite{CC19, DT20}) and simulated annealing (see e.g. \cite{KaSrZa09, LiProArZhGo09}). 
Because the isa has better ergodic properties than parallel tempering, there is big hope that the isa can produce even better results.
Additionally, our main results show that the isa mixes much faster than the over-damped Langevin dynamics. 
Therefore, one expects that the isa also outperforms the over-damped Langevin dynamics for simulated annealing. 
In this article, we show that this is indeed the case. 
From a computational point of view, one has to investigate the trade-off between the theoretical improvement and the cost of doubling the dimension of the underlying state space. Hence, further studies on the computational costs are needed to decide whether isa could practically compete with state-of-the-art methods for simulated annealing, e.g. methods based on L\'evy flights \cite{Pav07} or Cuckoo's search \cite{YaSu09}. 

There are a few directions to extend this work. 
From the Eyring-Kramers formulas for the isa, we obtain deviation upper bounds for the convergence to equilibrium at low temperatures. It is interesting to know whether these upper bounds are optimal, and to derive matching lower bounds. Also, we plan to extend the study of the isa to the underdamped Langevin dynamics, for which the Eyring-Kramers formula of the Poincar\'e constant was established in \cite{HHS11}. Furthermore, one could also extend the isa to L\'evy flights and apply it to simulated annealing for even better performance.

\medskip
{\bf Organization of the paper:} In Section \ref{s2}, we provide background, derive the isa, present the main results and apply these results to sampling and simulated annealing. 
In Section \ref{s_proofs}, we give proofs of the results stated in Section \ref{s2}.

\section{Setting, main results and applications}

\label{s2}
In this section, we start by discussing how the isa emerges as the weak limit from parallel tempering. 
Then we introduce the precise setting and assumptions. 
After this we present the main results of this article, the Eyring-Kramers formula for the Poincar\'e constant and an estimate of the log-Sobolev constant for the isa. 
We also give indications that they are optimal.
We close this section by discussing two applications: sampling Gibbs measures at low temperatures and simulated annealing.

\subsection{ISA as the weak limit of parallel tempering} \label{s_isa_as_limit_of_pt} 

Before describing parallel tempering and isa, let us first consider the over-damped Langevin equation which is a single diffusion specified by a sufficiently smooth, non-convex energy landscape $H: \mathbb{R}^n \rightarrow \mathbb{R}$ and a temperature $\tau > 0$.
It is governed by the SDE:
\begin{align}
\label{eq:odl}
d \xi_t &= -\nabla H(\xi_t) dt  + \sqrt{2 \tau} dB_t,
\end{align}
where $(B_t, \, t \ge 0)$ is standard Brownian motion in $\mathbb{R}^{n}$. 
The infinitesimal generator of the diffusion process~\eqref{eq:odl} is
\begin{align}
L_{\tau} := \tau \Delta - \nabla H \cdot \nabla.
\end{align}
Under some growth assumptions on $H$ (e.g. those of \cite[Section 1.2]{MS14}), the over-damped Langevin equation \eqref{eq:odl} has a unique invariant measure with density:
\begin{equation}\label{e_spGm_entire}
\nu^\tau(x) := \frac{1}{Z^\tau} \exp{\left(-\frac{H(x)}{\tau}\right)}
\end{equation} 
where $Z^\tau$ is the normalizing constant.
This probability measure is known as the Gibbs measure with energy landscape $H$ and temperature $\tau$.
The Dirichlet form associated with the Gibbs measure $\nu^{\tau}$ is defined for any suitable test function $f:\mathbb{R}^d \to \mathbb{R}$ by
\begin{align}
\mathcal{E}_{\nu^\tau} (f) := \int_{\mathbb{R}^n} (-L_\tau f) f d\nu^\tau = \int_{\mathbb{R}^n} \tau |\nabla f|^2 d \nu^\tau.
\end{align}
For general non-convex energy landscape $H$, the over-damped Langevin equation shows metastable behavior at low temperatures $\tau$ in the sense of a separation of time scales:
\begin{itemize}[itemsep = 3 pt]
	\item
	In the short run, the process converges fast to a local minimum of the energy landscape $H$;
	\item
	In the long run, the process stays near a local minimum for exponentially long time before it jumps to another local minimum.
\end{itemize}
In the previous work of \cite{MS14}, this behavior is captured by explicit, low-temperature asymptotic formulas (known as Eyring-Kramers formulas) for the two constants $\rho, \alpha > 0$ appearing in the following two functional inequalities for the invariant measure $\nu^\tau$: 
the Poincar\'e inequality (PI$(\rho)$)
\begin{align}
\label{eq:PIdef}
\var_{\nu^\tau}(f) := \int \Bigl(f - \int f d\nu^\tau\Bigr)^2 d\nu^\tau \leq \frac{1}{\rho} \mathcal{E}_{\nu^\tau} (f)
\end{align}
and the log-Sobolev inequality (LSI$(\alpha)$)
\begin{align}
\label{eq:LSIdef}
\ent_{\nu^\tau}(f^2) := \int f^2 \log \frac{f^2}{\int f^2 d\nu^\tau} d\nu^\tau \leq \frac{2}{\alpha} \mathcal{E}_{\nu^\tau} (f)
\end{align}
holding for all sufficiently smooth test functions $f: \mathbb{R}^{n} \rightarrow \mathbb{R}$. 

It is understood that for larger constants $\rho, \alpha > 0$, the diffusion process tends faster to equilibrium.
More precisely, the constants $\rho$ and $\alpha$ are the exponential rate of relaxation to equilibrium measured in variance or relative entropy, respectively.
Thus, it is useful to obtain lower bounds on the constants $\rho, \alpha$, or equivalently upper bounds on their inverse $\rho^{-1}, \alpha^{-1}$.
Also note that the Poincar\'e and the log-Sobolev inequalities \eqref{eq:PIdef}--\eqref{eq:LSIdef} are defined slightly different from those in \cite{MS14}, where $\mathcal{E}_{\nu^\tau} (f)$ is replaced with $\int |\nabla f|^2 d\nu^{\tau}$ on the right side.
Thus, the constants $\rho, \alpha$ defined by \eqref{eq:PIdef}--\eqref{eq:LSIdef} differ from those in \cite{MS14} up to a factor of $\tau$.

In the present work, we extend these results to an inhomogeneous diffusion, the ``infinite swapping process''. 
It arises from parallel tempering by swapping particle temperatures, which we now introduce. Given two temperatures $0 < \tau_1 < \tau_2 \ll 1$, $\tau_2 > K \tau_1 $ for some $K > 1$, define two product measures on $\mathbb{R}^{n} \times \mathbb{R}^n$:
\begin{equation}
\pi^+(x_1, x_2) := \nu^{\tau_{1}}(x_1)\nu^{\tau_{2}}(x_2),\qquad \pi^-(x_1, x_2) := \nu^{\tau_{2}}(x_1) \nu^{\tau_{1}}(x_2).
\end{equation} 
Identify the symbols $\sigma = +, -$ with the identity and the swap permutation on $\{1, 2\}$, respectively. 
Then $\pi^\sigma$ is the invariant measure of the following simple product SDE:
\begin{equation*}
\left\{ \begin{array}{rcl}
dX_1 = -\nabla H(X_1) \, dt + \sqrt{2 \tau_{\sigma(1)}} \, d B_1 \:, \\ 
dX_2 = -\nabla H(X_2) \, dt + \sqrt{2 \tau_{\sigma(2)}} \, d B_2 \:,
\end{array}\right.
\end{equation*}
where $B := (B_1, B_2)$ is standard Brownian motion in $\mathbb{R}^{n} \times \mathbb{R}^n$. 
Its infinitesimal generator consists of the two infinitesimal generators of the marginals
\begin{align}
L_\sigma := L_{\tau_{\sigma(1)}}^{x_1} + L_{\tau_{\sigma(2)}}^{x_2},
\end{align}
where the superscripts indicate the variable the generators are acting on. By construction $L_\sigma$ is reversible with respect to $\pi^\sigma$ and its associated Dirichlet form is
\begin{align}
\mathcal{E}_{\pi^\sigma} (f) := \int_{\mathbb{R}^n \times \mathbb{R}^n} (-L_\sigma f) f d\pi^\sigma =	\E_{\pi^\sigma} (\tau_{\sigma(1)} |\nabla_{x_1} f|^2 + \tau_{\sigma(2)} |\nabla_{x_2} f|^2). 
\end{align}
The idea of parallel tempering is to swap between the positions of $X_1$ and $X_2$. 
At some random times, $X_1$ is moved to the position of $X_2$ and vice-versa, so the resulting process is a Markov process with jumps.
To guarantee that the invariant measure remains the same, the jump intensity is of the Metropolis form $a\,g(x_1,x_2)$, where the constant `$a$' is the swapping rate of parallel tempering, and $g = \min \left(1, \pi^-/\pi^+ \right)$.
The resulting process is denoted by $(X_1^a(t), X_2^a(t))$.

Intuitively, larger values of `$a$' lead to faster convergence to equilibrium. However, the process $(X_1^a(t), X_2^a(t))$ is not tight so it does not converge weakly as $a \rightarrow \infty$. 
The key idea of \cite{Dupuis12} is to swap the \emph{temperatures} of $(X_1,X_2)$ instead of swapping the \emph{positions}.
More precisely, they consider the following process
\begin{equation*}
\left\{ \begin{array}{rcl}
d\overline{X}^a_1 = -\nabla H(X_1) \, dt + \sqrt{2 \tau_1 \mathds{1}_{Z^a = 0} + 2 \tau_2 \mathds{1}_{Z^a = 1}} \, d B_1 \:, \\ 
d\overline{X}_2 = -\nabla H(X_2) \, dt + \sqrt{2 \tau_2 \mathds{1}_{Z^a = 0} + 2 \tau_1 \mathds{1}_{Z^a = 1}} \, d B_2 \:,
\end{array}\right.
\end{equation*}
where $Z^a$ is a jump process which switches from state $0$ to state $1$ with intensity $a\, g(\overline{X}^a_1, \overline{X}^a_2)$, and from state $1$ to state $0$ with intensity $a\, g(\overline{X}^a_2, \overline{X}^a_1)$.
It was shown in \cite{Dupuis12} that as $a \rightarrow \infty$, the process $(\overline{X}^a_1(t), \overline{X}^a_2(t)$ converges weakly to the isa, whose dynamics is governed by the SDE \eqref{e_infinite_swapping}.
We rewrite it as
\begin{equation}
\label{eq:infswap}
\left\{ \begin{array}{rcl}
dX_1 = -\nabla H(X_1)\, dt + \sqrt{2 a_1(X_1, X_2)} \, d B_1 \:, \\ 
dX_2 = -\nabla H(X_2)\, dt + \sqrt{2 a_2(X_1, X_2)} \, d B_2 \:,
\end{array}\right.
\end{equation}
where the state-dependent diffusion coefficients $a_1, a_2 : \mathbb{R}^n \times \mathbb{R}^n \to [\tau_1,\tau_2]$ are given by
\begin{align*}
a_1 := \tau_1 \rho^+ + \tau_2 \rho^- ~ &\quad\mbox{ and }\quad a_2 := \tau_2 \rho^+ + \tau_1 \rho^- \\
\mbox{ with }\quad \rho^+ := \frac{\pi^+}{\pi^+ + \pi^-} ~ &\quad\mbox{ and }\quad \rho^- := \frac{\pi^-}{\pi^+ + \pi^-}.
\end{align*}
The infinitesimal generator of the isa \eqref{eq:infswap} is 
\begin{align}
\mathcal{L} := \rho^+ L_+ + \rho^- L_- = -\nabla H(x_1) \cdot \nabla_{x_1} - \nabla H(x_2) \cdot \nabla_{x_2} + a_1 \Delta_{x_1} + a_2\Delta_{x_2},
\end{align}
which is no longer the sum of two one-particle generators due to the full-space dependent diffusion coefficients $a_1,a_2$. 
A short calculations shows that $\mathcal{L}$ is self-adjoint with respect to the invariant symmetric measure
\begin{equation}\label{eq:invmeasure}
\mu: = \tfrac{1}{2}(\pi^+ + \pi^-).
\end{equation}
Let us note that the measure $\mu$ in~\eqref{eq:invmeasure} is generally not of product form, which contributes to the effectiveness of the sampling, at the expense of certain complications in our analysis. The Dirichlet form associated with $\mu$ is given by
\begin{align}
\mathcal{E}_\mu (f) := \int (-\mathcal{L} f) f d\mu = \frac{1}{2}\mathcal{E}_{\pi^+} (f) + \frac{1}{2} \mathcal{E}_{\pi^-} (f) = \int (a_1  | \nabla_{x_1} f|^2 + a_2  | \nabla_{x_2} f|^2)  d\mu.
\end{align}
We also define the Fisher information
\begin{align}\label{eq:Fisher}
\mathcal{I}_\mu (f^2) := 2 \mathcal{E}_\mu (f).
\end{align}

\subsection{Growth and nondegeneracy assumptions} 
We adopt the same assumptions on the energy landscape $H$ as in \cite[Section 1.2]{MS14}. 
These assumptions are standard in the study of metastability (see e.g. \cite{Bovier04, Bovier05}). 

\begin{definition}[Morse function]\label{defmorse}
	A smooth function $H: \mathbb{R}^n\to\mathbb{R}$ is a \emph{Morse function} if the
	Hessian $\nabla^2 H$ of $H$ is nondegenerate on the set of critical
	points. That implies, for some $1 \leq C_H < \infty$ holds
	\begin{equation}
	\label{emorse} \forall x \in\mathcal{S}:= \bigl\{z\in\mathbb{R}^n: \nabla H(z) = 0 \bigr
	\} \ :\qquad \frac
	{|\xi|}{C_H} \leq \bigl |\nabla^2 H(x)\xi|  \leq C_H |\xi|.
	\end{equation}
\end{definition}

We also make the following growth assumptions on the potential $H$
to ensure the existence of PI and LSI.

\begin{assump}[PI]\label{assumeenv}
	$H\in C^3(\mathbb{R}^n, \mathbb{R})$ is a nonnegative Morse function, such that for
	some constants $C_H>0$ and $K_H \geq0$ holds
		\begin{align}
		\label{assumegradsuperlinear} \liminf_{|x| \to\infty} |\nabla H(x)| &\geq
		C_H ,\\
		\label{assumegradlaplace} \liminf_{|x|  \to\infty} \bigl(|\nabla H(x)|^2 - \Delta H(x) \bigr) &\geq - K_H .
		\end{align}
		
\end{assump}

\begin{assump}[LSI]\label{assumeenvLSI}
	$H\in C^3(\mathbb{R}^n, \mathbb{R})$ is a nonnegative Morse function, such that for
	some constants $C_H>0$ and $K_H \geq0$ holds
		\begin{align}
		\label{assumelyapLSI} \liminf_{|x| \to\infty}\frac{|
			\nabla H(x)|^2 - \Delta
			H(x)}{|x|^2} &\geq C_H ,\\
		\label{assumehessLSI} \inf_{x} \nabla^2 H(x) &\geq -K_H \operatorname{Id}.
		\end{align}
\end{assump}
\begin{remark}\label{assumeremdiscussion}
{\em
	Assumption \ref{assumeenv} has the following consequences for
	the energy landscape $H$:
	\begin{itemize}[itemsep = 3 pt]
		\item
		The condition \eqref{assumegradsuperlinear} and $H(x) \geq0$
		ensures that $e^{-\frac{H}{\tau}}$ is integrable and can be normalized to a
		probability measure on $\mathbb{R}^n$ (see \cite[Lemma 3.14]{MS14}).
		Hence, the probability measures $\nu^\tau$ (and therefore $\pi^+, \pi^-$ and $\mu$) are
		well-defined.
		\item The Morse condition \eqref{emorse} together with the growth
		condition \eqref{assumegradsuperlinear} ensures that the set $\mathcal{S}$
		of critical points is discrete and finite. In particular, it follows
		that the set of local minima is a finite set $\mathcal{M} =  \{m_1,\dots
		,m_N \}$.
		\item Together with the rest of Assumption \ref{assumeenv}, the Lyapunov-type condition \eqref{assumegradlaplace} leads to a local PI for the Gibbs measures $\nu^\tau$ (see \cite[Theorem 2.9]{MS14}).
		\end{itemize}
		Similarly, Assumption \ref{assumeenvLSI} yields the following consequences for the energy landscape~$H$. 
		\begin{itemize}[itemsep = 3 pt]
		\item It leads to a local LSI for the Gibbs measures $\nu^\tau$ (see \cite[Theorem 2.10]{MS14}).
		\item Assumption \ref{assumeenvLSI} implies Assumption \ref{assumeenv}, which is natural in light of the fact that LSI is stronger than PI. 
	\end{itemize}
}	
\end{remark}
To keep the presentation clear, we also make some nondegeneracy assumptions on the energy landscape $H$. 
First, to simplify some formulas, we assume without loss of generality throughout that 
\begin{equation*}
\min_{x \in \mathbb{R}^n} H(x) = 0.
\end{equation*}
The saddle height $\widehat{H}(m_i,m_j)$ between two local minima $m_i, m_j$ is defined by
\begin{equation*}
\widehat{H}(m_i,m_j) : = \inf \left\{\max_{s \in [0,1]} H(\gamma(s)): \gamma \in C([0,1], R^n), \, \gamma(0) = m_i, \, \gamma(1) = m_j \right\}.
\end{equation*}
\begin{assump}
\label{assumption}
Let $m_1, \cdots, m_N$ be the positions of the local minima of $H$. 
\begin{enumerate}[label=(\roman*), itemsep = 3 pt]
\item
$m_1$ is the unique global minimum of $H$, and $m_1, \ldots, m_N$ are ordered in the sense that there exists $\delta > 0$ such that
\begin{equation}
\label{eq:nondeg}
H(m_N) \geq H(m_{N-1}) \geq \cdots \geq H(m_2) \ge \delta >0 = H(m_1).
\end{equation}
\item
For each $i, j \in [N]: = \{1, \ldots, N\}$, the saddle height between $m_i, m_j$ is attained at a unique critical point $s_{ij}$ of index one. That is, $H(s_{ij}) = \widehat{H}(m_i,m_j)$, and if $\{\lambda_1, \ldots, \lambda_n\}$ are the eigenvalues of $\nabla^2 H(s_{ij})$, then $\lambda_1 =: \lambda^- < 0$ and $\lambda_i > 0$ for $i \in \{2, \ldots, n\}$. 
The point $s_{ij}$ is called the communicating saddle point between the minima $m_i$ and $m_j$.
\item
There exists $p \in [N]$ such that the energy barrier $H(s_{p1}) - H(m_p)$ dominates all the others. That is,
there exists $\delta > 0$ such that for all $i \in [N] \setminus \{p\}$,
\begin{equation}
\label{eq:dombarrier}
E_* := H(s_{p1}) - H(m_p) \ge H(s_{i1}) - H(m_i) + \delta.
\end{equation}
The dominating energy barrier $E_*$ is called the critical depth.
\end{enumerate}
\end{assump}
\begin{figure}[h]
\centering
\setlength{\unitlength}{0.7\textwidth}
  \begin{picture}(1,0.52092241)%

	\put(0,0){\includegraphics[width=\unitlength]{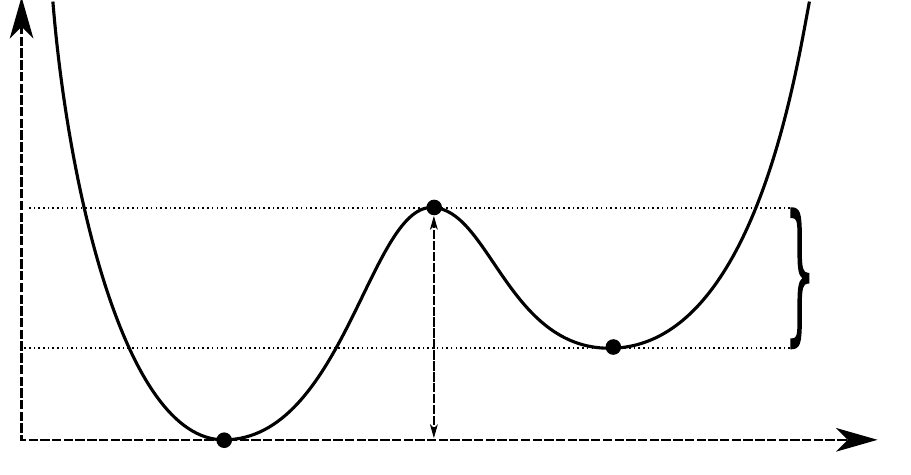}}%

	\put(0.425,0.305){\makebox(0,0)[lb]{\smash{saddle}}}%

	\put(0.43,0.08){\makebox(0,0)[lb]{\rotatebox{90}{\small{$\widehat{H}(m_1,m_2)$}}}}%
	\put(-0.002,0.02){\makebox(0,0)[lb]{\smash{0}}}%
	\put(0.9,0.2){\makebox(0,0)[lb]{\smash{$E_*$}}}%

	\put(0.58,0.1){\makebox(0,0)[lb]{\smash{local minima}}}%

	\put(0.65,0.155){\makebox(0,0)[lb]{\smash{$m_2$}}}%
	\put(0.14,-0.005){\makebox(0,0)[lb]{\smash{global minima}}}%

	\put(0.225,0.055){\makebox(0,0)[lb]{\smash{$m_1$}}}%
\end{picture}%
\caption{Illustration of the critical depth of a double-well function.}
\label{fig:CD}
\end{figure}

\subsection{The Eyring-Kramers formulas}
Our main results are the Eyring-Kramers formula for the Poincar\'e constant and a good estimate for log-Sobolev constant for the isa. Here a crucial new feature occurs in comparison to the over-damped Langevin dynamic: the lower temperature $\tau_1$ cannot be arbitrarily smaller than the higher temperature $\tau_2$ and there is an effective restriction on their ratio $\tau_1/\tau_2$. We comment on this observation in Subsection~\ref{s:necessary_lower_bound_tau1}. For ease of comparison, we begin by recalling the Eyring-Kramers formulas for the Poincar\'e and log-Sobolev constants for the Gibbs measure $\nu^{\tau}$, which is the invariant measure of a single diffusion at temperature $\tau$ governed by the over-damped Langevin equation \eqref{eq:odl}.

\begin{theorem}[Corollary 2.15 and 2.18 in \cite{MS14}] Assume $0 < \tau \ll 1$. Suppose that the energy landscape $H$ satisfies Assumptions \ref{assumeenv} and \ref{assumption}. Then the Gibbs measure $\nu^{\tau}$ satisfies the Poincar\'e inequality \eqref{eq:PIdef} with the constant $\rho$ satisfying
	\begin{align}\label{eq:PIconstant_single} \frac{1}{\rho} \leq \frac{1}{\rho^\tau}:=  \frac{2 \pi  \sqrt{|\det \nabla^2 H(s_{p1})|}}{\sqrt{|\det \nabla^2 H(m_p)|} |\lambda^{-}(s_{p1})|} \exp\biggl(\frac{H(s_{p1}) - H(m_p)}{\tau}\biggr) \left( 1 + O (\sqrt{\tau} \, \lvert\ln \tau\rvert^{\frac{3}{2}})  \right).
\end{align}
	Here $\lambda^{-}(s_{p1})$ is the negative eigenvalue of the Hessian $\nabla^2 H(s_{p1})$ at the communicating saddle point $s_{p1}$.
\end{theorem}

\begin{theorem}[Corollary 2.17 and 2.18 in \cite{MS14}] Assume $0 < \tau \ll 1$. Suppose that the energy landscape $H$ satisfies Assumptions \ref{assumeenvLSI} and \ref{assumption}. Then the Gibbs measure $\nu^{\tau}$ satisfies the log-Sobolev inequality \eqref{eq:LSIdef}
with the constant $\alpha$ satisfying 
\begin{equation}\label{eq:LSIconstant_single}
\frac{2}{\alpha} \leq \frac{2}{\alpha^\tau} := \left(\frac{H(m_p)}{\tau} + \log \sqrt{\frac{|\det \nabla^2 H(m_1)|}{|\det \nabla^2 H(m_p)|}} \right) \frac{1}{\rho^\tau}.
\end{equation}
where $\rho^\tau$ is defined in \eqref{eq:PIconstant_single}.
\end{theorem}

Now we are ready to state our main results.

\begin{theorem}[Eyring-Kramers formula for the Poincar\'e constant for the isa]
	\label{thm:PI}
	Assume that $\tau_2 \geq K \tau_1$ for some constant $K>1$. Let $\mu$ be the invariant measure of the isa defined by \eqref{eq:invmeasure}. Suppose that the energy landscape $H$ satisfies Assumptions \ref{assumeenv} and \ref{assumption}. Then the measure $\mu$ satisfies the Poincar\'e inequality
	\begin{equation}\label{eq:PI}
	\var_{\mu}(f) \le \frac{1}{\rho} \mathcal{E}_{\mu}(f)
	\end{equation}
	with the constant $\rho$ satisfying
	\begin{align}\label{eq:PIest}
	\frac{1}{\rho}  \leq \frac{1}{\rho^{\text{PI}}} := \frac{1}{\rho^{\tau_2}} + O(1) \Phi_n\Bigl(\frac{\tau_2}{\tau_1}\Bigr).
	\end{align}
	Here $\rho^{\tau_2}$ is given by the asymptotic formula \eqref{eq:PIconstant_single} with $\tau = \tau_2$, and $\Phi_n: [1, \infty) \rightarrow [0, \infty)$ is the function
	\begin{equation}\label{eq:temperature_dependence}
	\Phi_n(x) = \begin{cases} 1 & \text{if } n = 1, \\
	1 + \ln x & \text{if } n=2, \\
	1 + x^{(n-2)/2} & \text{if } n\geq 3.
	\end{cases}
	\end{equation}
\end{theorem}

\begin{theorem}[Estimate for the log-Sobolev constant of the isa]
	\label{thm:LSI}
	Assume that~$\tau_2 \geq K \tau_1$ for some constant~$K>1$. Let $\mu$ be the invariant measure of the isa defined by \eqref{eq:invmeasure}. Suppose that the energy landscape $H$ satisfies Assumptions \ref{assumeenvLSI} and \ref{assumption}.
	Then the measure $\mu$ satisfies the log-Sobolev inequality
	\begin{equation}
	\label{eq:LSI}
	\ent_{\mu}(f) := \int f \ln f \, d \mu - \int f \, d \mu \ln \int f \, d \mu \le \frac{1}{\alpha} \mathcal{I}_{\mu}(f),
	\end{equation}
	so that $\ent_{\mu}(f^2)\leq \frac{2}{\alpha} \mathcal{E}_{\mu}(f)$ with 
	\begin{align}
	\label{eq:LSIest}
	\frac{2}{\alpha} \leq \frac{2}{\alpha^{\text{LSI}}} := 2N^2  \left(\frac{H(m_p)}{\tau_1}+\frac{H(m_p)}{\tau_2}\right) \frac{1}{\rho^{\tau_2}} + O\left(\frac{1}{\tau_1}\right) \Phi_n\Bigl(\frac{\tau_2}{\tau_1}\Bigr).
	\end{align}
	Here $N$ is the number of local minima of $H$, $\rho^{\tau_2}$ is given by the asymptotic formula \eqref{eq:PIconstant_single} with $\tau = \tau_2$, and $\Phi_n$ is the function defined in \eqref{eq:temperature_dependence}.
\end{theorem}

A simple calculation shows that the terms involving $\Phi_n$ are asymptotically negligible compared to the rest of these formulas, provided $\tau_1$ is not too small compared to $\tau_2$:

\begin{corollary} Impose the condition that
	\begin{equation}\label{eq:condition_for_temperatures}
	\frac{1}{\tau_1} \leq \begin{cases}    \exp{\left(o\left(\frac{1}{\tau_2}\right)\right)} \quad & \text{if } n\geq 3, \\
\exp{\left(\exp{\left(o\left(\frac{1}{\tau_2}\right)\right)}\right)} \quad & \text{if }  n = 2.
	\end{cases}
	\end{equation}	
	Then, with the assumptions of Theorem \ref{thm:PI}, the measure $\mu$ satisfies the Poincar\'e inequality \eqref{eq:PI} with constant $\rho$ satisfying 
		\begin{equation}\label{eq:PIest2} \frac{1}{\rho}  \leq \frac{1}{\rho^{\tau_2}},
		\end{equation}
	and with the assumptions of Theorem \ref{thm:LSI}, the measure $\mu$ satisfies the log-Sobolev inequality \eqref{eq:LSI} with constant $\alpha$ satisfying 
	\begin{equation}\label{eq:LSIest2}\frac{2}{\alpha}  \leq 2N^2  \left(\frac{H(m_p)}{\tau_1}+\frac{H(m_p)}{\tau_2}\right) \frac{1}{\rho^{\tau_2}}.
	\end{equation}
	Here $\rho^{\tau_2}$ is given by the asymptotic formula \eqref{eq:PIconstant_single} with $\tau = \tau_2$.
\end{corollary}

\begin{remark}\label{r_ek_compared} \em{ Comparing the Eyring-Kramers formulas \eqref{eq:PIest2} and \eqref{eq:LSIest2} for the isa at temperatures $(\tau_1, \tau_2)$ to the corresponding formulas \eqref{eq:PIconstant_single} and \eqref{eq:LSIconstant_single} derived for a single diffusion at the lower temperature $\tau_1$, the main difference is that in the exponent $\frac{H(s_{p1}) - H(m_p)}{\tau_1}$, the lower temperature $\tau_1$ is now replaced by the higher temperature $\tau_2$ ,as long as $\frac{1}{\tau_1}$ grows sub-exponentially as $\frac{1}{\tau_2}$ in the limit $\tau_1, \tau_2 \rightarrow 0$. Since we assume $\tau_2\geq K\tau_1$ for some constant $K > 1$, this means the energy barrier $H(s_{p1}) - H(m_p)$ is effectually reduced by a factor of $K > 1$.
}

\end{remark}

\subsection{Dependence on the ratio between temperatures}\label{s:necessary_lower_bound_tau1}
The following proposition shows that the dependence on $\tau_2/\tau_1$ in the Poincar\'e and LSI constants of the isa is necessary and the function $\Phi_n$ that describes this dependence is nearly optimal. 
\begin{proposition}\label{p_opt_local} If $\tau_2, \tau_1/\tau_2$ are sufficiently small, then there exists a constant $C > 0$ and for every $\eta > 0$, there exists a constant $C_\eta > 0$, such that
	\begin{align*}
	\sup_{f\in H^1(\mu)} \frac{\var_{ \mu}(f)} {\mathcal{E}_{\mu}(f)} \geq \begin{cases} C_\eta (\tau_2/\tau_1)^{(1-\eta) (n-2)/2} & \text{for } n\geq 3, \\
	C \ln(\tau_2/\tau_1) & \text{for } n = 2.
	\end{cases} 
	\end{align*}
\end{proposition}

\subsection{Optimality of the Eyring-Kramers formulas in dimension one}

For the over-damped Langevin dynamics, the corresponding Eyring-Kramers formula for Poincar\'e inequality has been shown to be optimal. For the isa, the Poincar\'e constant of \eqref{eq:PIest} is optimal in a generic one-dimensional case. This gives a strong indication of optimality in higher dimensions.
For notational simplicity, we will henceforth write $A\lessapprox B$ if
\begin{equation}\label{eq:approx_notation}
A\leq B \left( 1 + O (\sqrt{\tau_2} \, \lvert\ln \tau_2\rvert^{\frac{3}{2}})  \right)
\end{equation}
i.e. up to multiplicative errors occurring in~\eqref{eq:PIconstant_single} with $\tau = \tau_2$. We write $A \approx B$ if $A\lessapprox B$ and $B\lessapprox A$. 

\begin{proposition}\label{p_opt_EK_1d} 
Assume $n = 1$, and $H$ has three critical points: two minima $m_1 < m_2$ with $H(m_1) = 0 < \delta \leq H(m_2)$ and a local maximum $s$ in between. Then
	\begin{align*}
	\sup_{f\in H^1(\mu)} \frac{\var_{ \mu}(f)} {\mathcal{E}_{\mu}(f)} \gtrapprox \frac{1}{\rho^{\PI}},
	\end{align*}
where $\rho^{\PI}$ is given by the asymptotic formula \eqref{eq:PIest} and $H^1(\mu):= \{f: \int_{\mathbb{R}^n} |\nabla f|^2 d\mu < \infty\}$
\end{proposition}
For the over-damped Langevin dynamics, the corresponding Eyring-Kramers formula for LSI inequality has been shown to be optimal in the one-dimensional case. For the isa, we do not expect the LSI constant of \eqref{eq:LSIest} to be optimal. However, up to the combinatorial pre-factor in the number of local minima $N$, it captures the asymptotic behavior for a generic one-dimensional case. 
\begin{proposition}\label{p_opt_EK_1d_LSI} 
	Assume $n = 1$, and $H$ has three critical points: two minima $m_1 < m_2$ with $H(m_1) = 0 < \delta \leq H(m_2)$ and a local maximum $s$ in between. Then
	\begin{align*}
	\sup_{f\in H^1(\mu)} \frac{\ent_{\mu}(f^2)} {\mathcal{I}_{\mu}(f^2)} \gtrapprox \frac{1}{\alpha^{\LSI}},
	\end{align*}
	where $\alpha^{\LSI}$ is given by the asymptotic formula \eqref{eq:LSIest}. 
\end{proposition}

\subsection{Application to sampling}

It is well known that estimates on the Poincar\'e and the log-Sobolev constant yield estimates for the rate of convergence to equilibrium of the underlying process. 
Applying this to the isa, we obtain the following direct consequence of Theorem \ref{thm:PI} and Theorem \ref{thm:LSI}. 
We refer to \cite[Theorem 1.7]{Schthesis} for a proof in the setting of the over-damped Langevin dynamics. 
The argument directly carries over to the isa.
\begin{corollary}
Let $f_t$ be the relative density of the isa \eqref{eq:infswap} at time $t$ with respect to the invariant measure $\mu$. 
\begin{enumerate}[itemsep = 3 pt]
\item[(i)]
Under the same assumptions as in Theorem \ref{thm:PI}, it holds that
\begin{equation*}
\var_{\mu}(f_t) \le e^{-2 \rho t} \var_{\mu}(f_0),
\end{equation*}
where $\rho$ satisfies the estimate \eqref{eq:PIest}.
\item[(ii)]
Under the same assumptions as in Theorem \eqref{thm:LSI}, it holds that
\begin{equation*}
\ent_{\mu}(f_t) \le e^{-2 \alpha t} \ent_{\mu}(f_0),
\end{equation*}
where~$\alpha$ satisfies the estimate \eqref{eq:LSIest}.
\end{enumerate}
\end{corollary}

Another well-known consequence is that the Poincar\'e or log-Sobolev constant allows to quantify the ergodic theorem i.e. to estimate speed of convergence of the time average to the ensemble mean. 
See \cite[Proposition 1.2]{CaGu08} and \cite[Corollary 4]{Wu00} for a proof in the setting of the over-damped Langevin dynamics. 
The same argument carries over to the isa.

\begin{corollary}\label{p_sampling}
Let $\nu$ denote the initial law of the isa \eqref{eq:infswap}.
\begin{enumerate}[itemsep = 3 pt]
\item
Under the same assumptions as in Theorem~\ref{thm:PI}, it holds that for all functions~$f: \mathbb{R}^n \times \mathbb{R}^n \to \mathbb{R}$ such that~$\sup|f|=1$, all~$0<R\leq 1$ and all $t>0$,
\begin{equation*}
  \mathbb{P}_{\nu}\!\left( \frac{1}{t} \int_0^t f(X_1(s), X_2(s)) \, ds - \int f \, d \mu \geq R \right) \leq \left\|  \frac{d \nu}{ d\mu} \right\|_{L^2} \exp\!\left(  - \frac{tR^2 \rho }{ 8  \var_{\mu}(f) }\right),
\end{equation*}
where~$\rho$ satisfies the estimate \eqref{eq:PIest}. 
\item
Under the same assumptions as in Theorem~\ref{thm:LSI}, it holds that for all functions~$f \in L^1(\mu)$ and all~$R,t>0$,
\begin{align*}
  \mathbb{P}_{\nu}\!\left( \frac{1}{t} \int_0^t f(X_1(s), X_2(s)) ds - \int f d \mu \geq R \right) \leq \left\|  \frac{d \nu}{ d\mu}\right\|_{L^2} \exp\bigl(  - t \alpha  H^* (R)\bigr),
\end{align*}
where~$\alpha$ satisfies the estimate \eqref{eq:LSIest} and
\begin{align*}
  H^*(R):= \sup_{\lambda \in \mathbb{R}} \left\{\lambda  R - \ln \int \exp\!\left( \lambda \Bigl( f - \int \! f \, d \mu \Bigr) \right) d \mu \right\}.
\end{align*}
\end{enumerate}
Similar bounds hold for the negative deviation. 
\end{corollary}
One consequence of Corollary \ref{p_sampling} is that the isa has an exponential gain in comparison with the over-damped Langevin dynamics for sampling (see also Remark \ref{r_ek_compared}). 
The deviation bounds show an explicit dependence of the convergence on the temperatures, which is missing in the large deviation analysis in \cite{Dupuis12}.
This justifies why the choice of a second higher temperature in the isa is useful, and shows how it increases the speed of convergence in the ergodic theorem.

\subsection{Application to simulated annealing}
\label{s4}

Here we apply the log-Sobolev inequality in Theorem \ref{thm:LSI} to the isa for simulated annealing.  

The goal of simulated annealing is to find the global minimum of a function~$H: \mathbb{R}^n \to \mathbb{R}$ that is potentially non-convex. Let us explain the main idea of the stochastic version of simulated annealing. One considers a stochastic process on~$H$ subject to thermal noise. 
When simulating this process, one lowers the temperature slowly over time. Hereby, the stochastic process gets trapped. Now, the goal is to show that the trapped process converges to the global minimum of~$H$ with high probability. This is typically true if the temperature is lowered slowly enough. Hence, another goal is to find the best stochastic process with the fastest possible cooling schedule that still allows to approximate the global minimum.

Simulated annealing adapted to the over-damped Langevin dynamics was studied in \cite{GH86, Miclo}, see also \cite{TZ21} for a review and results in discrete time.
As we will see below, the cooling schedule has to be logarithmically slow. 
This implies long waiting time in order to reach the global minimum. There are many approaches to improve this behavior. Luckily, one has the freedom to choose the underlying stochastic process used for simulated annealing. One of the most efficient approach is called Cuckoo search and is based on L\'evy flights (see \cite{Pav07,YaSu09}). Those methods are able to find the global minimum in certain situations with a polynomial cooling schedule. 
An alternative is to use replica exchange or parallel tempering. As we know from~\cite{Dupuis12}, mixing only improves when particles are swapped faster, making the isa a natural candidate for accelerating simulated annealing. 

In \cite{Miclo}, it was shown that for simulated annealing adapted to the over-damped Langevin dynamics, the fastest successful cooling schedule is characterized by the Eyring-Kramers formula for the log-Sobolev constant. 
However, no estimates on the associated log-Sobolev constant at low temperatures were known at that time. 
Hence, more sophisticated arguments were applied by \cite{HKS} to replace the log-Sobolev constant by the Poincar\'e constant showing that the fastest successful cooling schedule is characterized by the critical depth $E_*= H(s_{1p})-H(m_p)$. Only in 2014, the Eyring-Kramers formula for the log-Sobolev constant was derived in \cite{MS14} which leads to a more direct proof of the same result. 
This formula was then used by \cite{MM} to study simulated annealing adapted to the underdamped Langevin dynamics, showing that it is at least as good as simulated annealing adapted to the over-damped Langevin dynamics.
The main result of \cite{HKS, Miclo} is stated as follows.
\begin{theorem}[\cite{HKS, Miclo}]
  Let $(X_t, \, t \ge 0)$ be the process of simulated annealing adapted to the over-damped Langevin dynamics:
\begin{align}\label{e_over-damped_langevin}
  dX_t = - \nabla H(X_t) \, dt + \sqrt{2 \tau(t)} \, dB_t.
\end{align}
Let $E_*:= H(s_{p1})- H(m_p)$ denote the critical depth of the  energy landscape $H$. Then
\begin{enumerate}[itemsep = 3 pt]
\item[(i)] 
If $E \le \liminf_{t \to \infty} \tau(t) \ln t \le \limsup_{t \to \infty} \tau(t) \ln t < \infty$ with $E > E_*$, then for all $\delta >0$,
\begin{align*}
 \mathbb{P}\bigl( H(X_t) \leq H(m_1) + \delta \bigr) \rightarrow 1 \quad \mbox{as } t \to \infty.
\end{align*}
\item[(ii)]
If $\limsup_{t \to \infty} \tau(t) \ln t \le E$ with $0<E < E_* $, 
then for $\delta$ small enough,
\begin{align*}
  \limsup_{t \to \infty} \mathbb{P}\bigl( H(X_t) \leq H(m_1) + \delta \bigr) < 1.
\end{align*}
\end{enumerate}
\end{theorem}

Applying the isa to simulated annealing yields:
\begin{equation}
\label{eq:infswapSA}
\left\{ \begin{array}{rcl}
dX_1 = -\nabla H(X_1) \, dt + \sqrt{2 \, \tau_1(t) \, \rho(X_1,X_2) + 2 \, \tau_2(t) \, \rho(X_2,X_1)} \, d B_1, \\ 
dX_2 = -\nabla H(X_2) \, dt + \sqrt{2 \, \tau_2(t) \, \rho(X_1,X_2) + 2 \, \tau_1(t) \, \rho(X_2,X_1)} \, d B_2.
\end{array}\right.
\end{equation}
We require that for some fixed constant~$K>1$
\begin{equation*}
\tau_2(t) = K \tau_1(t) \qquad \text{and} \qquad \tau_1(t) \downarrow 0 \:.
\end{equation*}

In Theorem \ref{thm:PI} and Theorem \ref{thm:LSI}, we showed that the infinite swapping dynamics mixes faster than the over-damped Langevin dynamics. 
Choosing $\tau_2= K \tau_1$, the effective critical depth of the potential~$H$ is~$\frac{E_*}{K}$ compared to $E_*$ for simulated annealing adapted to the over-damped Langevin dynamics given by \eqref{e_over-damped_langevin}. 
This indicates that the infinite swapping dynamics could outperform the over-damped Langevin dynamics for simulated annealing. 
The main result of this section shows that this is true. 

\begin{theorem}\label{p_infswap_sa}
 Assume that the energy landscape $H$ satisfies Assumptions \ref{assumeenvLSI} and \ref{assumption}. 
 Let $E_*:= H(s_{p1})- H(m_p)$ be the critical depth of the energy landscape $H$. For $E > \frac{E_*}{K}, K > 1$, let
 \begin{equation}
 \label{eq:tau12t}
 \tau_1(t) = \frac{E}{ \ln (2+t)} \quad \mbox{and} \quad \tau_2(t) = \frac{K E}{\ln (2+t)}.
 \end{equation}
Let $X_1,X_2$ be given by~\eqref{eq:infswapSA} with initial distribution $m$. Let $m_t(x_1,x_2)$ be the probability density of $(X_1(t), X_2(t))$. Assume the following moment condition for the initial distribution $m$: for every $p\geq 1$, there exists a constant $C_p$ such that
\begin{equation}
\label{eq:momentest}
\int \bigl(H(x_1) + H(x_2)\bigr)^p dm(x_1,x_2) \leq C_p.
\end{equation} 
Then for all $\delta > 0$, $\varepsilon > 0$
	\begin{align}
	\label{eq:SAconcentration}
	\mathbb{P}(\min\{H(X_1(t)), H(X_2(t))\} > \delta) \lesssim \left(\frac{1}{1+t}\right)^{\min\left(  \frac{\delta}{E}, \frac{1}{2} - \frac{E_*}{2KE} \right) - \varepsilon}.
	\end{align}
\end{theorem}

\section{Proofs}\label{s_proofs}
\subsection{Proof of Theorem~\ref{thm:PI} and Theorem \ref{thm:LSI}} 

As in \cite{MS14}, we decompose the state space $\mathbb{R}^n$ into an ``admissible partition'' of metastable regions $\{\Omega_{i}\}_{i=1}^N$, as defined below.
\begin{definition}[Admissible partition]\label{d_admissable_single}
	The family $\{\Omega_{i}\}_{i=1}^N$ with $\Omega_{i}$ open and connected is called an admissible partition for $H$ if 
	\begin{enumerate}[label=(\roman*), itemsep = 3 pt]
		\item
		for each $i \in [N]$, the local minimum $m_{i} \in \Omega_{i}$,
		\item
		$\{\Omega_{i}\}_{i=1}^N$ forms a partition of $\mathbb{R}^n$ up to sets of Lebesgue measure zero,
		\item
		The partition sum of $\Omega_{i}$ is approximately Gaussian. That is, there exists $\tau_0 > 0$ such that for all $\tau < \tau_0$, for $i \in [N]$,
		\begin{align}
		\nu^\tau (\Omega_i) Z^\tau &= \int_{\Omega_i}  \exp\biggl(-\frac{H(x)}{\tau}\biggr) dx \nonumber\\
		&= \frac{(2\pi \tau)^{n/2}}{\sqrt{\det \nabla^2 H (m_i)}} \exp\biggl(-\frac{H(m_i)}{\tau}\biggr)(1+O(\sqrt{\tau}|\log \tau|^{3/2})). \label{e_spGM_admissible}
		\end{align}
	\end{enumerate}
\end{definition}
\begin{remark} 
	{\em A canonical way to obtain an admissible partition for $H$ is to associate to each local minimum $m_{i}$ for $i\in [N]$
		its basin of attraction with respect to the gradient flow of~$H$. That is,
		\begin{equation}
		\Omega_{i} =  \left\{y \in \mathbb{R}^N: \lim_{t \rightarrow \infty} y_t = m_i, \, \frac{dy_t}{dt} = -\nabla H(y_t), \, y_0 = y \right\}.
		\end{equation}
		However, as in \cite{MS14}, to facilitate the proof, we choose instead the basins of attraction for the gradient flow of a suitable perturbation of $H$
		(see Section \ref{s_prooflocal}).
	}
\end{remark}

Suppose $\{\Omega_i\}_{i=1}^N$ is an admissible partition in the sense of Definition \ref{d_admissable_single}. Define local measures on $\mathbb{R}^n$
\begin{align}
\nu_i^\tau(x) &:= \frac{1}{Z_i^\tau} \nu^\tau(x)|_{\Omega_i}, \label{e_spGm_local} \\
Z_i^\tau &:= \nu^\tau (\Omega_i) \approx \frac{\sqrt{\det \nabla^2 H (m_1)}}{\sqrt{\det \nabla^2 H (m_i)}} \exp{\left(-\frac{H(m_i)}{\tau}\right)} (1+O(\sqrt{\tau}|\ln \tau|^{3/2})).
\end{align}
This induces a decomposition of the measure $\mu$ on $\mathbb{R}^n \times \mathbb{R}^n$ as
\begin{align}\label{eq:mixture}
	\mu = \frac{1}{2}(\pi^+ + \pi^-)=  \sum \frac{1}{2} Z_{ij}^+ \pi^+_{ij} +  \sum \frac{1}{2} Z_{ij}^- \pi^- _{ij}
\end{align}
where for $1 \le i,j \le n$, $Z_{ij}^+ := Z_i^{\tau_1} Z_j^{\tau_2}, 
Z_{ij}^- := Z_i^{\tau_2} Z_j^{\tau_1}$ and
\begin{align}\pi^+_{ij}(x_1, x_2) := \frac{1}{Z_{ij}^+} \pi^+ (x_1, x_2) |_{\Omega_i \times \Omega_j} = \nu_i^{\tau_1} (x_1) \nu_j^{\tau_2}(x_2)  \\
 \pi^-_{ij}(x_1, x_2) := \frac{1}{Z_{ij}^-}\pi^-(x_1, x_2) |_{\Omega_i \times \Omega_j} = \nu_i^{\tau_2} (x_1) \nu_j^{\tau_1}(x_2).
\end{align}
The following results are read from \cite[Lemma 2.4 and Corollary 2.8]{MS14}.
\begin{lemma}[Decomposition of variance]\label{decompvar_new}
	For the mixture representation~\eqref{eq:mixture} of the Gibbs measure $\mu$, and a smooth function $f: \mathbb{R}^n \times \mathbb{R}^n \rightarrow \mathbb{R}$, it holds
	\begin{align}
		\label{eq:localvar}
		\var_{\mu}(f) &= \frac{1}{2} \sum Z_{ij}^+ \var_{\pi^+_{ij}}(f) + \frac{1}{2} \sum Z_{ij}^- \var_{\pi^-_{ij}}(f) \\
		\label{eq:meandiff_transport}
		&\,\,\, + \frac{1}{4} \sum Z_{ij}^+ Z_{kl}^+ (\E_{\pi^+_{ij}}(f) - \E_{\pi^+_{kl}}(f))^2 + \frac{1}{4} \sum Z_{ij}^- Z_{kl}^- (\E_{\pi^-_{ij}}(f) - \E_{\pi^-_{kl}}(f))^2 \\
		\label{eq:meandiff_switch}
		&\,\,\, + \frac{1}{4} \sum Z_{ij}^+ Z_{kl}^- (\E_{\pi^+_{ij}}(f) - \E_{\pi^-_{kl}}(f))^2.
	\end{align}
	where the second line is summing over unordered pairs $\{(i,j), (k,l)\}$ and the last line is summing over ordered pairs $((i,j), (k,l))$. 
\end{lemma}

\begin{lemma}[Decomposition of entropy]\label{decompent_new}
	For the mixture representation~\eqref{eq:mixture} of the Gibbs measure $\mu$, and a smooth function $f: \mathbb{R}^n \times \mathbb{R}^n \rightarrow \mathbb{R}$, it holds
	\begin{align}
	\label{eq:localent}
	\ent_{\mu}(f^2) &\leq \frac{1}{2} \sum Z_{ij}^+ \ent_{\pi^+_{ij}}(f^2) + \frac{1}{2} \sum Z_{ij}^- \ent_{\pi^-_{ij}}(f^2) \\
	\label{eq:localvar_ent}
	&\,\,\, + \frac{1}{2} \sum_{(i, j)} \left(\sum_{(k, l) \neq (i, j)} \frac{ Z_{kl}^+}{\Lambda(Z_{ij}^+, Z_{kl}^+)} + \sum_{(k, l)} \frac{ Z_{kl}^-}{\Lambda(Z_{ij}^+, Z_{kl}^-)}\right) Z_{ij}^+\var_{\pi^+_{ij}}(f) \\
	\label{eq:localvar_ent_2}
	&\,\,\, + \frac{1}{2} \sum_{(i, j)} \left(\sum_{(k, l) \neq (i, j)} \frac{ Z_{kl}^-}{\Lambda(Z_{ij}^-, Z_{kl}^-)} + \sum_{(k, l)} \frac{ Z_{kl}^+}{\Lambda(Z_{ij}^-, Z_{kl}^+)}\right) Z_{ij}^-\var_{\pi^-_{ij}}(f) \\
	\label{eq:meandiff_transport_ent}
	&\,\,\, + \frac{1}{2} \sum_{\sigma \in\{-,+\}}\sum \frac{Z_{ij}^\sigma Z_{kl}^\sigma}{\Lambda(Z_{ij}^\sigma, Z_{kl}^\sigma)} (\E_{\pi^\sigma_{ij}}(f) - \E_{\pi^\sigma_{kl}}(f))^2 \\
	\label{eq:meandiff_switch_ent}
	&\,\,\, + \frac{1}{2} \sum \frac{Z_{ij}^+ Z_{kl}^-}{\Lambda(Z_{ij}^+, Z_{kl}^-)} (\E_{\pi^+_{ij}}(f) - \E_{\pi^-_{kl}}(f))^2,
\end{align}
	where the second to last line is summing over unordered pairs $\{(i,j), (k,l)\}$ and the last line is summing over ordered pairs $((i,j), (k,l))$. Here the function $\Lambda:[0,\infty)\times [0,\infty)  \to [0,infty)$ is the logarithmic mean defined by
	\begin{equation}\label{eq:def:logMean}
		\Lambda(a,b) = \int_0^1 a^{(1-s)} b^s \, ds = \begin{cases}
			\frac{a-b}{\log a - \log b} , & a \ne b ;\\
			a , & a= b  .
		\end{cases}
	\end{equation}
\end{lemma}
The local variances appearing in \eqref{eq:localvar}, \eqref{eq:localvar_ent} and \eqref{eq:localvar_ent_2}  and the local entropies appearing in \eqref{eq:localent} are treated by the Poincar\'e and the log-Sobolev inequalities for local product measures. 
\begin{lemma}[Local PI for $\pi^\sigma_{ij}$]\label{p_PI_prod} Under Assumption \ref{assumeenv} and given $\tau_2$ small enough, 
there exists an admissible partition $\{\Omega_i\}_{i=1}^N$ such that for all $\tau \leq \tau_2$, , for all smooth functions $f: \mathbb{R}^n \times \mathbb{R}^n \rightarrow \mathbb{R}$
	\begin{equation}
	\var_{\pi^\sigma_{ij}}(f) \overset{\eqref{e_prod_var}}{\leq}   O(1) \E_{\pi^\sigma_{ij}}(\tau_{\sigma(1)}  |\nabla_{x_1} f|^2 + \tau_{\sigma(2)}  |\nabla_{x_2} f|^2). 
	\end{equation}
\end{lemma}
\begin{lemma}[Local LSI for $\pi^\sigma_{ij}$]\label{p_LSI_prod} Under Assumption \ref{assumeenvLSI}, for all smooth functions $f: \mathbb{R}^n \times \mathbb{R}^n \rightarrow \mathbb{R}$
	\begin{equation}
	\ent_{\pi^\sigma_{ij}}(f^2) \overset{\eqref{e_prod_ent}}{\leq}   O(1) \E_{\pi^\sigma_{ij}}( |\nabla_{x_1} f|^2 + |\nabla_{x_2} f|^2).
	\end{equation}
\end{lemma}
We defer the details of the proof of Lemmas \ref{p_PI_prod} and \ref{p_LSI_prod} to Section \ref{s_prooflocal}. They are based on the simple product structure of the measures $\pi^\sigma_{ij}$ and an adaption of the local Poincar\'e inequality \cite[Theorem 2.9]{MS14} and the local LSI inequality \cite[Theorem 2.10]{MS14}.
In the sequel, for a Dirichlet form $\mathcal{E}(f)$, we denote $\mathcal{E}(f)[\Omega]$ to be the Dirichlet integral with region of integration restricted to $\Omega$. \medskip
It follows that
\begin{align}\label{eq:localvar_estimate}
Z_{ij}^\sigma \var_{\pi^\sigma_{ij}}(f) &\leq  O(1) \mathcal{E}_{\pi^\sigma}(f)[\Omega_i \times \Omega_j], \\
\label{eq:localent_estimate}
Z_{ij}^\sigma \ent_{\pi^\sigma_{ij}}(f) &\leq  O(\tau_1^{-1}) \mathcal{E}_{\pi^\sigma}(f)[\Omega_i \times \Omega_j].
\end{align}

To deal with the mean-differences appearing in \eqref{eq:meandiff_transport} and \eqref{eq:meandiff_transport_ent}, we will apply the mean-difference estimate from \cite[Theorem 2.12]{MS14}, which allows us to transport in one of the variables $x_1, x_2$ at a time from one metastable region $\Omega_j$ to another metastable region $\Omega_k$. 
In order to ensure that we only get exponential dependence on $1/\tau_2$ rather than $1/\tau_1$ in the Eyring-Kramers formulas, we only transport in the high-temperature variable, and not in the low-temperature variable. This allows us to deal with mean-differences of the type between $\pi^+_{ij}$ and $\pi^+_{ik}$, or the type between $\pi^-_{ji}$ and $\pi^-_{ki}$.

\begin{lemma}[Mean-difference estimates for $\pi^+_{ij}, \pi^+_{ik}$ and for $\pi^-_{ji}, \pi^-_{ki}$] Recall the notation $\approx, \lessapprox$ defined in \eqref{eq:approx_notation}. We have
	\begin{align}\label{eq:meandiff_transport_plus}
	Z_{ik}^+ (\E_{\pi^+_{ij}} f - \E_{\pi^+_{ik}} f)^2 
	&\lessapprox C_{jk}^{\tau_2}\cdot  \mathcal{E}_{\pi^+}(f)[\Omega_i \times \mathbb{R}^n], \\
	\label{eq:meandiff_transport_minus}
	Z_{ki}^- (\E_{\pi^-_{ji}} f - \E_{\pi^-_{ki}} f)^2 &\lessapprox C_{jk}^{\tau_2}\cdot \mathcal{E}_{\pi^-}(f)[ \mathbb{R}^n\times \Omega_i].
	\end{align}
	where 
	\[
	C_{jk}^{\tau_2} := \frac{2\pi \sqrt{\det \nabla^2 H(s_{jk})}}{\sqrt{\det \nabla^2 H(m_k)}|\lambda^-(s_{jk})|} \exp{\left(\frac{H(s_{jk}) - H(m_k)}{\tau_2}\right)}.
	\]
\end{lemma}

\begin{proof} For the first estimate, applying Cauchy-Schwarz and \cite[Theorem 2.12]{MS14}, we get
\begin{align}
	Z_{ik}^+ (\E_{\pi^+_{ij}} f - \E_{\pi^+_{ik}} f)^2 &\leq Z_i^{\tau_1} Z_k^{\tau_2} \E_{\nu^{\tau_1}_{i}} (\E_{\nu^{\tau_2}_{j}} f - \E_{\nu^{\tau_2}_{k}} f)^2 \\
	&\lessapprox Z_i^{\tau_1}\E_{\nu^{\tau_1}_{i}} C_{jk}^{\tau_2} \int \tau_2 |\nabla_{x_2} f|^2 d\nu^{\tau_2}(x_2)\\
	&\leq C_{jk}^{\tau_2}\cdot  \mathcal{E}_{\pi^+}(f)[\Omega_i \times \mathbb{R}^n].
\end{align}
The second estimate is completely analogous. 
\end{proof}

To deal with the mean-differences in \eqref{eq:meandiff_switch} and \eqref{eq:meandiff_switch_ent}, we have another move available, which is to swap the temperatures of the two variables, i.e. to swap between $\pi^+_{ij}$ and $\pi^-_{ij}$. This is the main new technical ingredient compared to \cite{MS14}, which comes at a cost of a term involving the ratio of the higher temperature to the lower temperature, $\tau_2/\tau_1$. 

\begin{lemma}[Mean-difference estimate for $\pi^+_{ij}, \pi^-_{ij}$]\label{p_meandiff_prod}
	\begin{align}
	(\E_{\pi^+_{ij}} f - \E_{\pi^-_{ij}} f)^2 &\leq  \Phi_n\Bigl(\frac{\tau_2}{\tau_1}\Bigr)  O(\tau_2) (\E_{\pi^+_{ij}} |\nabla_{x_2} f|^2 + \E_{\pi^-_{ij}} |\nabla_{x_1} f|^2 ) \\
	&\quad + \omega(\tau_2) \sum_{\sigma\in\{+,-\}} \E_{\pi^\sigma_{ij}} (\tau_{\sigma(1)}  |\nabla_{x_1} f|^2 + \tau_{\sigma(2)} |\nabla_{x_2} f|^2). 
	\end{align}
	for any smooth function $f: \mathbb{R}^n \times \mathbb{R}^n \rightarrow \mathbb{R}$, where $\Phi_n$ is the function defined in equation \eqref{eq:temperature_dependence} and $\omega(\tau_2) := O(\sqrt{\tau_2} |\log \tau_2|^{3/2})$.
\end{lemma}

We defer the proof of this lemma to Section \ref{sc:34}. It follows that
\begin{align}\label{eq:meandiff_switch_estimate}
\min(Z_{ij}^+, Z_{ij}^-) (\E_{\pi^+_{ij}} f - \E_{\pi^-_{ij}} f)^2 \leq \Phi_n\Bigl(\frac{\tau_2}{\tau_1}\Bigr) O(1) \mathcal{E}_{\mu}(f).
\end{align}

Using these estimates, we will show that the dominating terms in Lemma \ref{decompvar_new} are the mean-differences between $\pi_{ip}^+, \pi_{11}^+$ and between $\pi_{pj}^-, \pi_{11}^-$ where $i, j$ are arbitrary and $p$ is the local minimum with the dominating energy barrier. 

\begin{lemma}\label{lem:meandiff_dominating} Let $p$ be the local minimum with the dominating energy barrier. Then for any $i, j\in [N]$, and $\sigma \in\{+, -\}$
	\begin{align}
	Z_{ip}^+Z_{11}^\sigma (\E_{\pi^+_{ip}}(f) - \E_{\pi^\sigma_{11}}(f))^2 &\lessapprox C_{1p}^{\tau_2}\cdot \mathcal{E}_{\pi^+}(f)  [\Omega_i\times \mathbb{R}^n] + \Phi_n\Bigl(\frac{\tau_2}{\tau_1}\Bigr) O(1) \mathcal{E}_{\mu}(f) \\
	Z_{pj}^-Z_{11}^\sigma (\E_{\pi^-_{pj}}(f) - \E_{\pi^\sigma_{11}}(f))^2 &\lessapprox C_{1p}^{\tau_2}\cdot \mathcal{E}_{\pi^-}(f)  [\mathbb{R}^n \times \Omega_j] + \Phi_n\Bigl(\frac{\tau_2}{\tau_1}\Bigr) O(1) \mathcal{E}_{\mu}(f) 
	\end{align}
Moreover, if $\{(i, j)^{\sigma_1}, (k, l)^{\sigma_2}\}$ is one of the following forms
$$\{(i, 1)^+, (1, 1)^+\}, \{(1, j)^-, (1, 1)^-\}, \{(i, 1)^+, (1, 1)^-\}, \{(1, 1)^+, (1, l)^-\}, $$
then
\begin{align}
Z_{ij}^{\sigma_1} Z_{kl}^{\sigma_2} (\E_{\pi^{\sigma_1}_{ij}}(f) - \E_{\pi^{\sigma_2}_{kl}}(f))^2 \leq \Phi_n\Bigl(\frac{\tau_2}{\tau_1}\Bigr) O(1) \mathcal{E}_{\mu}(f)
\end{align}
Finally, for any other $\{(i, j)^{\sigma_1}, (k, l)^{\sigma_2}\}$, the term $Z_{ij}^{\sigma_1} Z_{kl}^{\sigma_2} (\E_{\pi^{\sigma_1}_{ij}}(f) - \E_{\pi^{\sigma_2}_{kl}}(f))^2$ is negligible in the sense of being exponentially smaller in $1/\tau_2$ compared to one of the terms above on the right hand side. 
\end{lemma}

\begin{proof} Let $\Gamma$ be the graph whose vertices are labelled $\cdot_{ij}^\sigma$ and have three kinds of edges:
	\begin{itemize}
		\item ``vertical'' edges between $\cdot_{ij}^+, \cdot_{ik}^+$
		\item ``horizontal'' edges between $\cdot_{ij}^-, \cdot_{kj}^-$
		\item ``swapping'' edges between $\cdot_{ij}^+, \cdot_{ij}^-$
	\end{itemize}
We decompose the mean-difference between any two measures $\pi^+_{ij}, \pi^-_{kl}$ as a sum of mean-differences of the types in \eqref{eq:meandiff_transport_plus}, \eqref{eq:meandiff_transport_minus}, and \eqref{eq:meandiff_switch_estimate}, corresponding to a sequence of ``moves'' on the graph $\Gamma$. Given any sequence of moves $v_0 \rightarrow v_1 \rightarrow \cdots \rightarrow v_m$ on graph $\Gamma$, we have
\begin{align}
Z_{v_0} Z_{v_m} (\E_{\pi^{v_{0}}} f - \E_{\pi^{v_m}} f)^2 
&= Z_{v_0} Z_{v_m} \left(\sum_{t=1}^m \sqrt{\omega_t} \frac{1}{\sqrt{\omega_t}} (\E_{\pi^{v_{t-1}}} f - \E_{\pi^{v_{t}}} f) \right)^2 \\
&\le \sum_{t=1}^m \frac{1}{\omega_t} Z_{v_0} Z_{v_m} (\E_{\pi^{v_{t-1}}} f - \E_{\pi^{v_{t}}} f)^2 \label{eq:moveseqsum}
\end{align}
for any $\omega_t > 0, \sum_{t=1}^m \omega_t = 1$. After taking into account the weights $Z_{ij}^+, Z_{kl}^-$, this leads to the choice of the following three types of sequences of moves for the three types of mean-differences occurring in Lemma \ref{decompvar_new}: 
	\begin{itemize}
		\item Type I sequence: $\cdot_{ij}^+ \rightarrow \cdot_{i1}^+ \rightarrow \cdot_{i1}^- \rightarrow \cdot_{11}^- \rightarrow \cdot_{k1}^- \rightarrow \cdot_{k1}^+ \rightarrow \cdot_{kl}^+$
		\item Type II sequence: $\cdot_{ij}^- \rightarrow \cdot_{1j}^- \rightarrow \cdot_{1j}^+ \rightarrow \cdot_{11}^+ \rightarrow \cdot_{1l}^+ \rightarrow \cdot_{1l}^- \rightarrow \cdot_{kl}^-$
		\item Type III sequence: $\cdot_{ij}^+ \rightarrow \cdot_{i1}^+ \rightarrow \cdot_{i1}^- \rightarrow \cdot_{11}^- \rightarrow \cdot_{11}^+ \rightarrow \cdot_{1l}^+ \rightarrow \cdot_{1l}^- \rightarrow \cdot_{kl}^-$
	\end{itemize}
Let us first look at the decomposition \eqref{eq:moveseqsum} for a Type I sequence. For the 1st move, 
\begin{align}
Z_{ij}^{+} Z_{kl}^{+} (\E_{\pi^{+}_{ij}}(f) - \E_{\pi^{+}_{i1}}(f))^2 \lessapprox Z_{kl}^+ C_{j1}^{\tau_2} \cdot \mathcal{E}_{\pi^+}(f)[\Omega_i \times \mathbb{R}^n]
\end{align}
which is negligible unless $j=p, k=l=1$. For the 2nd move, 
\begin{align}
Z_{ij}^{+} Z_{kl}^{+} (\E_{\pi^{+}_{i1}}(f) - \E_{\pi^{-}_{i1}}(f))^2 \leq Z_j^{\tau_2} Z_{kl}^+ \cdot  \Phi_n \left(\frac{\tau_2}{\tau_1} \right) O(1)  \mathcal{E}_{\mu}(f)
\end{align}	
which is negligible unless $j=k=l=1$. For the 3rd move, 
\begin{align}
Z_{ij}^{+} Z_{kl}^{+} (\E_{\pi^{-}_{i1}}(f) - \E_{\pi^{-}_{11}}(f))^2 \lessapprox e^{-H(m_i)\bigl(\frac{1}{\tau_1} - \frac{1}{\tau_2}\bigr)} Z_j^{\tau_2} Z_{kl} C_{1i}^{\tau_2} \cdot \mathcal{E}_{\pi^-}(f)[\mathbb{R}^n \times \Omega_1]
\end{align}	
which is always negligible. The analysis for the remaining three moves are completely symmetric: the 4th move is always negligible, the 5th move is negligible unless $i=j=l=1$, and the 6th move is negligible unless $l=p, i=j=1$. \medskip 

Overall, if $(i, j), (k, l)$ is not one of the exceptions mentioned, we can just assign $\omega_1 = \omega_1 = \cdots = \omega_6 = 1/6$, then the overall sum is negligible. This choice of $(\omega_t)_{t=1}^6$ also works in the exceptional cases $k=j=l=1$ and $i=j=l=1$ (since we can afford to lose a constant factor because of the $O(1)$). \medskip

Lastly, in the exceptional case $j=p, k=l=1$, we consider a shortened 2-move sequence $\cdot_{ip}^+ \rightarrow \cdot_{i1}^+ \rightarrow \cdot_{11}^+$. For the 1st move in this sequence, 
\begin{align}
Z_{ip}^{+} Z_{11}^{+} (\E_{\pi^{+}_{ij}}(f) - \E_{\pi^{+}_{i1}}(f))^2 \lessapprox C_{p1}^{\tau_2} \cdot \mathcal{E}_{\pi^+}(f)[\Omega_i \times \mathbb{R}^n] 
\end{align}
and for the 2nd move in this sequence,
\begin{align}
Z_{ip}^{+} Z_{11}^{+} (\E_{\pi^{+}_{i1}}(f) - \E_{\pi^{+}_{11}}(f))^2 &\approx Z_p^{\tau_2} \cdot Z_{i1}^+ Z_{11}^+ (\E_{\pi^{+}_{i1}}(f) - \E_{\pi^{+}_{11}}(f))^2 \\
&\lessapprox Z_p^{\tau_2} \cdot \Phi_n\Bigl(\frac{\tau_2}{\tau_1}\Bigr) O(1) \mathcal{E}_{\mu}(f)
\end{align}
Thus, for this sequence, we can assign $\omega_1 = 1-Z_p^{\tau_2} \approx 1, \omega_2 = Z_p^{\tau_2}$, then the overall sum is as claimed. The exceptional case $l=p, i=j=1$ is completely symmetric. \medskip

The analysis for Type II and Type III sequences are completely analogous. 
\end{proof}

We can adapt this approach to estimate the terms in Lemma \ref{decompent_new}. 

\begin{lemma}\label{lem:meandiff_dominating_LSI} Let $p$ be the local minimum with the dominating energy barrier. Then for $i, k, l \in [N]$ and $\sigma \in \{+, -\}$ such that 
	$$H(m_i) < H(m_p) \mbox{ or } i = p, \mbox{ and } \frac{H(m_i)}{\tau_1} + \frac{H(m_p)}{\tau_2}
	\geq 
	\frac{H(m_k)}{\tau_{\sigma(1)}} + \frac{H(m_l)}{\tau_{\sigma(2)}}, $$

	\begin{align}
	\frac{Z_{ip}^{+}Z_{kl}^{\sigma}}{\Lambda(Z_{ip}^{+}, Z_{kl}^{\sigma})} (\E_{\pi^+_{ip}}(f) - \E_{\pi^\sigma_{kl}}(f))^2 & \lessapprox \frac{1}{\Lambda(\frac{Z_{ip}^+}{Z_{kl}^\sigma}, 1)} \left(C_{1p}^{\tau_2} \mathcal{E}_{\pi^+}(f)[\Omega_i \times \mathbb{R}^n]  + \Phi_n\Bigl(\frac{\tau_2}{\tau_1}\Bigr) O(1) \mathcal{E}_{\mu}(f)\right) \\
			\frac{Z_{pi}^{-}Z_{kl}^{\sigma}}{\Lambda(Z_{pi}^{-}, Z_{kl}^{\sigma})} (\E_{\pi^-_{pi}}(f) - \E_{\pi^\sigma_{kl}}(f))^2 & \lessapprox \frac{1}{\Lambda(\frac{Z_{pi}^-}{Z_{kl}^\sigma}, 1)} \left(C_{1p}^{\tau_2} \mathcal{E}_{\pi^-}(f)[\mathbb{R}^n \times \Omega_i]  + \Phi_n\Bigl(\frac{\tau_2}{\tau_1}\Bigr) O(1) \mathcal{E}_{\mu}(f)\right)
					\end{align}
Finally, for any other $\{(i, j)^{\sigma_1}, (k, l)^{\sigma_2}\}$, the term $\displaystyle \frac{Z_{ij}^{\sigma_1} Z_{kl}^{\sigma_2}}{\Lambda(Z_{ij}^{\sigma_1}, Z_{kl}^{\sigma_2})} (\E_{\pi^{\sigma_1}_{ij}}(f) - \E_{\pi^{\sigma_2}_{kl}}(f))^2$ is negligible in the sense of being exponentially smaller in $1/\tau_2$ compared to one of the terms above on the right hand side. 
\end{lemma}

\begin{proof} The analysis is similar as in the previous lemma, but now we have to take into account the logarithmic mean, using the estimate
	\begin{align}
	\frac{ab}{\Lambda(a, b)} = a\cdot \frac{b}{\Lambda(a/b, 1)} \lessapprox a \log (1/a)
	\end{align}
	for $b \lessapprox 1, a \ll 1$. The main difference is that we now need to be more careful to show the transport from $\cdot_{ip}^+$ to $\cdot_{11}^+$ is negligible if $H(m_i) \geq H(m_p)$ and $i\neq p$ by choosing the alternative path: $\cdot_{ip}^+ \rightarrow \cdot_{ip}^- \rightarrow \cdot_{1p}^- \rightarrow \cdot_{1p}^+ \rightarrow \cdot_{11}^+$. 
\end{proof}

\begin{proof}[Proof of Theorem \ref{thm:PI}] Combining Lemma \ref{decompvar_new}, \eqref{eq:localvar_estimate} and Lemma \ref{lem:meandiff_dominating}, we get
	\begin{align}
	\var_\mu(f) &\lessapprox \frac{1}{2} \sum_{i, j} O(1) \mathcal{E}_{\pi^+} (f) [\Omega_i\times \Omega_j] + \frac{1}{2} \sum O(1) \mathcal{E}_{\pi^-} (f) [\Omega_i\times \Omega_j] \\
	&\quad + 2\cdot \frac{1}{4}\sum_{i} C_{p1}^{\tau_2}\cdot \mathcal{E}_{\pi^+}(f)  [\Omega_i\times \mathbb{R}^n] + 2\cdot \frac{1}{4}\sum_{j} C_{1p}^{\tau_2}\cdot \mathcal{E}_{\pi^-}(f)  [\mathbb{R}^n \times \Omega_j] \\ 
	&\quad+ \Phi_n\Bigl(\frac{\tau_2}{\tau_1}\Bigr)  O(1) \mathcal{E}_{\mu}(f) \\
	&\leq \left(O(1) + C_{1p}^{\tau_2} + \Phi_n\Bigl(\frac{\tau_2}{\tau_1}\Bigr)  O(1)\right) \mathcal{E}_{\mu}(f) 
	\end{align}
as desired. 

\end{proof}

\begin{proof}[Proof of Theorem \ref{thm:LSI}] Combining Lemma \ref{decompent_new}, \eqref{eq:localvar_estimate}, \eqref{eq:localent_estimate} and Lemma \ref{lem:meandiff_dominating_LSI}, we get
	\begin{align}
	\MoveEqLeft\ent_\mu(f) \lessapprox \frac{1}{2} \sum_{i, j} O(\tau_1^{-1}) \mathcal{E}_{\pi^+} (f) [\Omega_i\times \Omega_j] + \frac{1}{2} \sum O(\tau_1^{-1}) \mathcal{E}_{\pi^-} (f) [\Omega_i\times \Omega_j] \\
	&\,\,\, +  \frac{1}{2} \sum_{i, j} 2N^2 O(\tau_1^{-1}) \cdot O(1) \mathcal{E}_{\pi^+} (f) + \frac{1}{2} \sum_{i, j} 2N^2 O(\tau_1^{-1}) \cdot O(1) \mathcal{E}_{\pi^-} (f) \\
	&\,\,\, + \frac{1}{2}\sum_{i\leq p} \left(\sum_{\sigma}\sum_{(k, l)} \frac{1}{\Lambda(\frac{Z_{ip}^+}{Z_{kl}^\sigma}, 1)} \right) \left(C_{1p}^{\tau_2}\cdot \mathcal{E}_{\pi^+}(f)  [\Omega_i\times \mathbb{R}^n] + \Phi_n\Bigl(\frac{\tau_2}{\tau_1}\Bigr)  O(1) \mathcal{E}_{\mu}(f) \right) \\
	&\,\,\, + \frac{1}{2}\sum_{i\leq p} \left(\sum_{\sigma}\sum_{(k, l)} \frac{1}{\Lambda(\frac{Z_{pi}^-}{Z_{kl}^\sigma}, 1)} \right) \left(C_{1p}^{\tau_2}\cdot \mathcal{E}_{\pi^-}(f)  [\mathbb{R}^n \times \Omega_j] + \Phi_n\Bigl(\frac{\tau_2}{\tau_1}\Bigr)  O(1) \mathcal{E}_{\mu}(f) \right)\\ 
	&\leq 2N^2 \left(O(\tau_1^{-1}) + H(m_p) (\tau_1^{-1} + \tau_2^{-1}) C_{1p}^{\tau_2} + O(\tau_1^{-1}) \Phi_n\Bigl(\frac{\tau_2}{\tau_1}\Bigr)     \right)\mathcal{E}_{\mu}(f)
	\end{align}
as desired. 
\end{proof}

\subsection{Proof of Theorem~\ref{p_infswap_sa}} 

With the help of Theorem \ref{thm:LSI}, i.e. the low-temperature asymptotics for the log-Sobolev constant, the proof of Theorem~\ref{p_infswap_sa} follows the arguments in \cite{Miclo, MM}. \smallskip

For each $t > 0$, let $\mu_t$ be the probability measure given in \eqref{eq:invmeasure} at temperatures $\tau_1 = \tau_1(t), \tau_2 = \tau_2(t)$ as defined in~\eqref{eq:tau12t}, i.e. $\mu_t(x_1,x_2) = \frac{1}{2}(\pi_t(x_1,x_2) + \pi_t(x_2,x_1))$, with
\begin{equation*}
\pi_t(x_1,x_2) : = \frac{1}{Z_t} \exp\!\left(-\frac{H(x_1)}{\tau_1(t)} - \frac{H(x_2)}{\tau_2(t)} \right),
\end{equation*}
where $Z_t$ is the normalizing constant. Our first observation is that the mass of the instantaneous equilibrium $\mu_t$ concentrates around the global minimum $\min H = 0$ as $t\rightarrow \infty$. 
\begin{lemma}\label{lem:sa_mu_t} If $(\tilde X_1(t), \tilde X_2(t))$ has law $\mu_t$, then for every $0 < \varepsilon < \delta$, there exists a constant $C>0$ such that
	\begin{equation*}
	\mathbb{P} (\min \{H(\tilde X_1(t)), H(\tilde X_2(t))\} > \delta) \leq C e^{-\frac{\delta - \varepsilon}{\tau_1(t)}}\leq C (2+t)^{-\frac{\delta - \varepsilon}{E}}.
	\end{equation*}
\end{lemma}
\begin{proof} Since $\mu_t(x_1, x_2) = \frac{1}{2}(\pi_t(x_1, x_2) + \pi_t(x_2, x_1))$, and $\min (H(x_1), H(x_2))$ is symmetric, 
	\begin{align*}
	\mathbb{P} (\min \{H(\tilde X_1(t)), H(\tilde X_2(t))\} > \delta) &= \mathbb{P} (\min \{H(\tilde Y_1), H(\tilde Y_2)\} >  \delta) \\
	&= \mathbb{P} (H(\tilde Y_1) >  \delta) \mathbb{P} (H(\tilde Y_2) >  \delta)  \\
	&\leq \mathbb{P} (H(\tilde Y_1) >  \delta),
	\end{align*}
	where $(\tilde Y_1, \tilde Y_2)$ has law $\pi_t$, and $\tilde Y_1, \tilde Y_2$ are independent. It remains to bound
	\begin{equation*}
	\mathbb{P} (H(\tilde Y_1) >  \delta) = \frac{\int_{H(x) > \delta}  e^{-\frac{H(x)}{\tau_1}} dx }{\int  e^{-\frac{H(x)}{\tau_1}} dx }.
	\end{equation*}
	Under Assumption \ref{assumeenvLSI}, \cite[Lemma 3.14]{MS14} applies and shows $H$ has linear growth at infinity. More specifically, there exists a constant $C_H$ such that for all sufficiently large $R$, 
	\[
	H(x) \geq \min_{|z|=R} H(z) + C(|x|-R) \,\,\, \mbox{ for } |x| > R.
	\] 
	In the above, we can choose $R$ large enough so that $\min_{|z|=R} H(z) > \delta$. Then 
	\begin{align*}\int_{H(x) > \delta}  e^{-\frac{H(x)}{\tau_1}} dx &= \int_{H(x) > \delta, |x| < R}  e^{-\frac{H(x)}{\tau_1}} dx + \int_{|x|>R}  e^{-\frac{H(x)}{\tau_1}} dx \\
	&\leq  e^{-\frac{\delta}{\tau_1}}\left(|B_R(0)|  + \int_{|x|>R}  e^{-\frac{ C(|x|-R)}{\tau_1}} dx\right) \\
	&\leq e^{-\frac{\delta}{\tau_1}} (|B_R(0)| + O(\tau_1)).
	\end{align*}
	On the other hand, there exists $r > 0$ such that $H(x) < \varepsilon$ when $|x| < r$. Then
	\begin{align*} \int  e^{-\frac{H(x)}{\tau_1}} dx &> \int_{|x| < r}  e^{-\frac{H(x)}{\tau_1}} dx > e^{-\frac{\varepsilon}{\tau_1}} |B_r(0)|.
	\end{align*}
	Combining these gives the desired estimate.
\end{proof}

Let $(\tilde X_1(t), \tilde X_2(t))$ be a random vector with law $\mu_t$. 
By Lemma \ref{lem:sa_mu_t} and Pinsker's inequality, we have
\begin{align}
\mathbb{P}(\min \{H(X_1(t)), H(X_2(t))\} >  \delta)  
& \leq \mathbb{P}(\min \{H(\tilde X_1(t)), H(\tilde X_2(t))\} >  \delta) + d_{TV}(\mu_t, m_t)  \\
&\leq C(2+t)^{-\frac{\delta - \varepsilon}{E}} + \sqrt{2 \ent (m_t | \mu_t) }, \label{eq:sa_pinsker}
\end{align}
where
\begin{equation*}
\ent (m_t | \mu_t) := \int \frac{m_t}{\mu_t} \ln \left( \frac{m_t}{\mu_t}\right) d\mu_t
\end{equation*}
is the relative entropy of $m_t$ with respect to $\mu_t$. Thus, it remains to bound $\ent (m_t | \mu_t)$. 
The following lemma gives an estimate of  $\frac{d}{dt}\ent (m_t | \mu_t)$, the proof of which is in the same spirit of \cite[Proposition 3]{Miclo}. 
\begin{lemma}
\label{lem:dIt}
It holds with $\mathcal{I}_{\mu}(\cdot)$ defined in~\eqref{eq:Fisher} the estimate
\begin{align}\label{eq:dIformula}
\frac{d}{dt} \ent (m_t | \mu_t) \leq - 2 \mathcal{I}_{\mu_t}\left( \frac{m_t}{\mu_t}\right) + \frac{d}{dt}\left(\frac{1}{\tau_1(t)} +\frac{1}{\tau_2(t)} \right) \E [H(X_1(t))  + H(X_2(t))].
\end{align}
\end{lemma}

\begin{proof} First note that
\begin{align}
\frac{d}{dt} \ent (m_t | \mu_t)  &= \int \frac{dm_t}{dt} \ln\!\left( \frac{m_t}{\mu_t}\right) dx + \int m_t \frac{d}{dt} \ln\!\left(\frac{m_t}{\mu_t} \right)dx \notag\\
& = \int \frac{dm_t}{dt} \ln\!\left( \frac{m_t}{\mu_t}\right) dx + \int \frac{dm_t}{dt} dx - \int \frac{m_t}{\mu_t} \frac{d\mu_t}{dt} dx \notag\\
& = \int \frac{dm_t}{dt} \ln\!\left( \frac{m_t}{\mu_t}\right) dx - \int \frac{d \ln(\mu_t)}{dt} d m_t.
\label{eq:dI}
\end{align}
We consider the first term in \eqref{eq:dI}. Observe that $m_t$ satisfies the Fokker-Planck equation
\begin{equation*}
\frac{dm_t}{dt} = \nabla_{x_1} \cdot (m_t \nabla_{x_1} H) + \nabla_{x_2} \cdot (m_t \nabla_{x_2} H) + \Delta_{x_1}(a_1 m_t) + \Delta_{x_2}(a_2 m_t).
\end{equation*}
Combining this with the identity $\nabla_{x_i}(a_i \mu_t) = - \mu_t \nabla_{x_i} H$, we get
\begin{equation*}
\frac{dm_t}{dt} = \nabla_{x_1} \cdot \left(a_1 \mu_t \nabla_{x_1} \left(\frac{m_t}{\mu_t}\right) \right) + \nabla_{x_2} \cdot \left(a_2 \mu_t \nabla_{x_2} \left(\frac{m_t}{\mu_t}\right) \right).
\end{equation*}
Integrating by parts, we have
\begin{align}
\int \frac{dm_t}{dt} \ln\!\left( \frac{m_t}{\mu_t}\right) dx &= - \int \left( a_1 \left|\nabla_{x_1}\left(\frac{m_t}{\mu_t} \right) \right|^2 + a_2 \left|\nabla_{x_2}\left(\frac{m_t}{\mu_t} \right) \right|^2\right) \frac{\mu_t}{m_t} d\mu_t \notag\\
&= - 2 \mathcal{I}_{\mu_t}\left( \frac{m_t}{\mu_t}\right),
\label{eq:1stterm}
\end{align}
where $\mathcal{I}_{\mu_t}$ is the Fisher information defined in \eqref{eq:Fisher} for $\mu = \mu_t$. Next we consider the second term in \eqref{eq:dI}. Using that $\min H = 0$ and that $\tau_1(t), \tau_2(t)$ are decreasing, direct calculation yields
\begin{align*}
-\frac{d\ln(\mu_t)}{dt} &\leq  \frac{d}{dt}\left(\frac{1}{\tau_1(t)} \right)\bigl(H(x_1) \rho(x_1, x_2) + H(x_2) \rho(x_2, x_1)\bigr) \\
& \ \ \ +\frac{d}{dt}\left(\frac{1}{\tau_2(t)} \right)\bigl(H(x_1) \rho(x_2, x_1) + H(x_2) \rho(x_1, x_2)\bigr) \\
&\leq \frac{d}{dt}\left(\frac{1}{\tau_1(t)} + \frac{1}{\tau_2(t)} \right)\bigl(H(x_1) + H(x_2)\bigr).
\end{align*}
Integrating this against $d m_t$ and combining it with \eqref{eq:1stterm} yields \eqref{eq:dIformula}.
\end{proof}
The second term on the right hand side of \eqref{eq:dIformula} are controlled via the following lemma.
\begin{lemma}\label{lem:momentest} For any $\varepsilon > 0$, there exists a constant $C$ such that
	\begin{equation*}
	\E\bigl[H(X_1(t)) + H(X_2(t))\bigr] \leq C (1 + t)^\varepsilon.
	\end{equation*}
\end{lemma}

We omit the proof of Lemma \ref{lem:momentest}, which closely follows that of \cite[Lemma 2]{Miclo}, using the moment assumptions on the initial distribution $m$ given by \eqref{eq:momentest} and growth assumptions on the energy landscape $H$ in Assumption \ref{assumeenvLSI}. 

\begin{lemma}\label{lem:sa_ent} For any $\varepsilon > 0$, there exists $C$ such that
	\begin{equation*}
	\ent (m_t | \mu_t) \leq C \left(\frac{1}{1+3t}\right)^{1 - \frac{E^*}{KE} - \varepsilon }.
	\end{equation*}
\end{lemma}
\begin{proof} Using the log-Sobolev inequality in Theorem \ref{thm:LSI}, the estimate \eqref{eq:dIformula} becomes
		\begin{equation*}
		\frac{d}{dt} \ent (m_t | \mu_t) \leq -2\alpha_t \ent (m_t | \mu_t) + \frac{2}{E} (2+t)^{-1} \E [H(X_1(t)) + H(X_2(t))],
		\end{equation*}
where $\alpha_t$ is the LSI constant in \eqref{eq:LSI} for $\mu = \mu_t$. From \eqref{eq:LSIest} we see that for any $\varepsilon > 0$, there exists $t_0 > 0$ and $C_1 > 0$ such that for $t > t_0$, 
\[
2\alpha_t \geq C_1 (2+t)^{-\frac{E_*}{KE}-\varepsilon}.\]
Together with Lemma \ref{lem:momentest}, we get that for $t > t_0$,
\begin{equation*}
\frac{d}{dt} \ent (m_t | \mu_t) \leq -C_1 (1+t)^{-\frac{E_*}{E}-\varepsilon} \ent (m_t | \mu_t)  + C_2 (1+t)^{-1+\varepsilon}.
\end{equation*}
A standard Gronwall-type argument as in the proof of ~\cite[Lemma 19]{MM} then finishes off the estimate. For $0< \varepsilon < \frac{1}{2}\left(1 - \frac{E_*}{KE}\right)$, let
\begin{equation*} Q(t) = \ent (m_t | \mu_t) - \frac{2C_2}{C_1} (1+t)^{-1+\frac{E_*}{KE}+2\varepsilon}.
\end{equation*}
Then for $t_0$ large enough and $t > t_0$,
\begin{align*}
\frac{d}{dt} Q(t)&\leq -C_1(1+t)^{-\frac{E_*}{KE}-\varepsilon} Q(t), \\
Q(t) &\leq Q(t_0) \exp{\left(-C_1 \int_{t_0}^t (1+s)^{-\frac{E_*}{KE}+\varepsilon} \, ds \right)}, \\
\ent (m_t | \mu_t) &\leq \frac{2C_2}{C_1} (1+t)^{-1+\frac{E_*}{KE}+2\varepsilon} + \ent (m_{t_0} | \mu_{t_0}) \exp{\left(-\frac{C_1}{\nu} ((1+t)^\beta - (1+t_0)^\beta)\right)},
\end{align*}
where $\beta := 1 - \frac{E_*}{KE} - \varepsilon > 0$, and the conclusion follows.
\end{proof}

Combining \eqref{eq:sa_pinsker} and Lemma \ref{lem:sa_ent}, we get that for any $\delta > 0, \varepsilon > 0$, there exists a constant $C$ such that 
\begin{equation*}
\mathbb{P}\Bigl(\min\bigl\{H(X_1(t)), H(X_2(t))\bigr\} > \delta \Bigr) \leq C\Biggl( \left(\frac{1}{1+t}\right)^{\frac{\delta - \varepsilon}{E}}  +  \left(\frac{1}{1+t}\right)^{\frac{1}{2} \left(1 - \frac{E^*}{KE} - \varepsilon \right)}\Biggr),
\end{equation*}
which implies \eqref{eq:SAconcentration}. 

\subsection{Proof of Lemmas \ref{p_PI_prod} and \ref{p_LSI_prod}}\label{s_prooflocal} The following decomposition of variance and entropy for a product measure reduces proving Lemmas \ref{p_PI_prod} and \ref{p_LSI_prod} to proving corresponding estimates for the component measures $\nu_i^\tau$. 

\begin{lemma}[Variance and entropy for product measure] Let $\pi = \nu_i \otimes \nu_j$ be a product of two probability measures on open subsets of $\mathbb{R}^n$.  For any smooth function $f: \mathbb{R}^n \times \mathbb{R}^n \rightarrow \mathbb{R}$
	\begin{align}
	\var_{\pi}(f) = \E_{\nu_j}\bigl(\var_{\nu_i}(f)\bigr) + \var_{\nu_j}\bigl( \E_{\nu_i} (f)\bigr) \leq \E_{\nu_j}\bigl(\var_{\nu_i}(f)\bigr) + \E_{\nu_i}\bigl( \var_{\nu_j}  (f)\bigr) \label{e_prod_var}.
	\end{align}
	For any smooth function $g: \mathbb{R}^n \times \mathbb{R}^n \rightarrow \mathbb{R}_{>0}$, \begin{align}	
	\ent_{\pi}(g) = \E_{\nu_j}\bigl(\ent_{\nu_i}(g)\bigr) + \ent_{\nu_j}\bigl(\E_{\nu_i}   (g)\bigr) \leq \E_{\nu_j}\bigl(\ent_{\nu_i}(g)\bigr) + \E_{\nu_i}\bigl(\ent_{\nu_j}(g)\bigr) \label{e_prod_ent}.
	\end{align}
\end{lemma}

\begin{definition}[Local PI and LSI for $\nu_i^\tau$]  The local Gibbs measure $\nu_i^\tau$ defined in~\eqref{e_spGm_local} satisfies a Poincar\'e inequality with constant $\rho$ if for all smooth functions $f: \mathbb{R}^n \rightarrow \mathbb{R}$
	\[
	\var_{\nu_i^\tau} (f) \leq \frac{1}{\rho} \E_{\nu_i^\tau} |\nabla f|^2,
	\]
	which is denoted by PI($\rho$). Likewise, $\nu_i^\tau$, defined in~\eqref{e_spGm_local}, satisfies a log-Sobolev inequality with constant $\alpha$ if for all smooth functions $f: \mathbb{R}^n \rightarrow \mathbb{R}$
	\[
	\ent_{\nu_i^\tau} (f^2) \leq \frac{2}{\alpha}\E_{\nu_i^\tau} |\nabla f|^2,
	\]
	which is denoted by LSI($\alpha$). 
\end{definition}
\begin{lemma}[Local PI for $\nu_i^{\tau}$]\label{p_PI_spGm} Under Assumption \ref{assumeenv}, given $\tau_2$ small enough, there exists an admissible partition $\{\Omega_i\}_{i=1}^N$ such that for all $\tau \leq \tau_2$, the local Gibbs measures $\nu^{\tau}_i$ satisfy PI($\rho$) with $\rho^{-1} = O(\tau)$.
\end{lemma}

\begin{lemma}[Local LSI for $\nu_i^{\tau}$]\label{p_LSI_spGm} Under Assumption \ref{assumeenvLSI}, given $\tau_2$ small enough, for the same admissible partition $\{\Omega_i\}_{i=1}^N$, for all $\tau \leq \tau_2$, the local Gibbs measures $\nu^{\tau}_i$ satisfy satisfy LSI($\alpha$) with $\alpha^{-1} = O(1)$.
\end{lemma}

Lemmas \ref{p_PI_spGm} and \ref{p_LSI_spGm}  are very similar to \cite[Theorem 2.9]{MS14} and \cite[Theorem 2.10]{MS14}, except now that we have two temperatures $\tau_1 < \tau_2$, we want the regions $\Omega_i$ in the admissible partition only depend on the higher temperature $\tau_2$ but not the lower temperature $\tau_1$, so that we can get PI and LSI for the local Gibbs measures $\nu_i^{\tau_1}, \nu_i^{\tau_2}$ at different temperatures \textbf{in the same regions} $\Omega_i$.  

This can be shown by making a small modification to the proof of \cite[Theorem 2.9, 2.10]{MS14}, which is based on constructing a Lyapunov function. Let us recall the definition of a Lyapunov function and the criterion for PI based on it from \cite{MS14}. 

\begin{definition}[Lyapunov function, Definition 3.7 in \cite{MS14}]\label{d_lya} A smooth function $W_\tau: \Omega_i \rightarrow (0, \infty)$ is a Lyapunov function for $\nu^\tau_i$ if for $L_\tau := \tau \Delta - \nabla H\cdot \nabla$
	\begin{enumerate} 	
		\item[(i)] There exists an open set $U_i\subset \Omega_i$ and constants $b > 0, \lambda > 0$ such that 
		\begin{equation}\frac{L_\tau W_\tau}{W_\tau} \leq -\lambda + b \mathbbm{1}_{U_i} \quad \forall x\in \Omega_i. \label{e_lya}
		\end{equation}
		\item[(ii)] $W_\tau$ satisfies Neumann boundary condition on $\Omega_i$ in the sense that it satisfies the integration by parts formula
		\begin{equation} \int_{\Omega_i} (-L_\tau W_\tau) g d\nu^\tau_i = \int_{\Omega_i} \nabla g \cdot \nabla W_\tau d\nu^\tau_i. 
		\label{e_ibp}
		\end{equation}
	\end{enumerate}
\end{definition}

\begin{lemma}[Lyapunov condition for local PI, Theorem 3.8 in \cite{MS14}]\label{p_lya} If there exists a Lyapunov function for $\nu^\tau_i$ in the sense of Definition \ref{d_lya} and that the truncated Gibbs measure $\nu^\tau_i|_{U_i}$ satisfies PI($\rho_{U_i}$), then the local Gibbs measure $\nu^\tau_i$ satisfies PI($\rho$) with 
	\[
	\rho^{-1} \leq \frac{b}{\lambda} \rho_{U_i}^{-1} + \frac{1}{\lambda} \tau.
	\]
\end{lemma}	
We choose $U_i$ to be a ball centered at the local minimum $m_i$ with a small, fixed radius $R_0$ such that $H$ is strongly convex on $U_i$. Then the Bakry-Emery criterion provides the following result.
\begin{lemma}[PI for truncated Gibbs measure, Lemma 3.6 in \cite{MS14}]  The measures $\nu^\tau_i|_{U_i}$ satisfy PI($\rho_{U_i}$) with $\rho_{U_i}^{-1} = O(\tau)$.  
\end{lemma}
In \cite{MS14}, the candidate for the Lyapunov function is $W_\tau = \exp\bigl({\frac{H}{2\tau}}\bigr)$, so that (see \cite[equation (3.9)]{MS14})
\[
\frac{L_\tau W_\tau}{W_\tau} = \frac{1}{2} \Delta H(x) - \frac{1}{4\tau} |\nabla H(x)|^2.
\]
In order to satisfy the condition \eqref{e_lya}, the Hamiltonian $H$ was replaced by a perturbed one $H_\tau$ such that $\lVert H- H_\tau \rVert_\infty = O(\tau)$. In order to satisfy the condition \eqref{e_ibp}, $\Omega_i$ is then chosen to be a basin of attraction with respect to the gradient flow of this perturbed Hamiltonian $H_\tau$. Consequently, the local PI was first deduced for the perturbed Gibbs measure $\frac{1}{Z} \exp{\frac{H_\tau}{2\tau}}$ on $\Omega_i$, which then implies PI for the original measure via Holley-Stroock perturbation principle. One side effect of this approach is that the region $\Omega_i$ depends on the temperature $\tau$, which is unsuitable in our setting with two different temperatures. 

We modify this approach as follows: instead of perturbing the Hamiltonian in the Gibbs measure, we only perturb the Hamiltonian in the Lyapunov function. Given $\tau_2 = \varepsilon$ small enough, we will choose a perturbation $H_\varepsilon = H + V_\varepsilon$ where $V_\varepsilon = O(\varepsilon)$, and choose $\Omega_i$ to be the basin of attraction with respect to the gradient flow of $H_\varepsilon$. Then, for every $\tau \leq \varepsilon$, we choose the Lyapunov function to be $W_\tau = \exp{\frac{H_\varepsilon}{2\tau}}$. Then \eqref{e_ibp} is satisfied by \cite[Theorem B.1]{MS14} and
\begin{align*}
\frac{L_\tau W_\tau}{W_\tau} &= -\frac{\nabla H\cdot \nabla H_\varepsilon}{2\tau} + \tau \left( \frac{\Delta H_\varepsilon}{2\tau} + \frac{|\nabla H_\varepsilon|^2}{4\tau^2} \right) \\
&= \frac{1}{2} \Delta H_\varepsilon - \frac{1}{4\tau} \bigl(|\nabla H|^2 - |\nabla V_\varepsilon|^2\bigr) \leq \frac{L_\varepsilon W_\varepsilon }{W_\varepsilon},
\end{align*}
where the last inequality holds as long as $|\nabla V_\varepsilon| \leq |\nabla H|$. Then once \eqref{e_lya} is verified for $\tau = \varepsilon$, PI for $\nu_i^\tau$ follows for every $\tau \leq \varepsilon$ on the same region $\Omega_i$. 

It turns out the same perturbation used in \cite{MS14} works here. Let $\mathcal{S}$ be the set of critical points of $H$ and $\mathcal{M} = \{m_1, m_2, \ldots,m_N \}$ be the set of local minima of $H$. 

\begin{lemma}[$\varepsilon$-modification]\label{p_epsilon} Given a function $H$ satisfying Assumption \ref{assumeenv}, there exist constants $\varepsilon_0, \lambda_0, a, C \in (0, \infty)$ and a family of $C^3$ functions $\{V_\varepsilon\}_{0< \varepsilon < \varepsilon_0}$ such that for $H_\varepsilon := H + V_\varepsilon$ it holds
	\begin{enumerate}	
		\item[(i)] $V_\varepsilon$ is supported on $\bigcup_{s\in \mathcal{S} \setminus \mathcal{M}} B_{a\sqrt{\varepsilon}}(s)$ and $|V_\varepsilon (x)| \leq C \varepsilon$ for all $x$. 
		\item[(ii)] Lyapunov-type condition: $|\nabla V_\varepsilon(x)| \leq |\nabla H(x)|$ for all $x$ and \begin{equation*}
		\frac{1}{2}\Delta H_\varepsilon - \frac{1}{4\varepsilon} (|\nabla H|^2 - |\nabla V_\varepsilon|^2) \leq -\lambda_0  \, \, \, \text{ for all }  x \notin  \bigcup_{m\in \mathcal{M}} B_{a\sqrt\varepsilon}(m). 
		\end{equation*}
	\end{enumerate} 
\end{lemma}

We omit the proof of Lemma \ref{p_epsilon}. It can be shown by carefully following the proof of \cite[Lemma 3.12]{MS14}; indeed, the perturbation $V_\varepsilon$ can be taken to be the same one used there. It is easy to see that $H_\varepsilon$ has the same local minima as $H$. For each local minimum $m_i$ of $H$, let $\Omega_i$ be the associated basin of attraction w.r.t. the gradient flow defined by the $\tau_2$-modified potential $H_{\tau_2}$, that is
\[
\Omega_i := \left\{y \in \mathbb{R}^n: \lim_{t \rightarrow \infty} y_t = m_i, \, \frac{dy_t}{dt} = -\nabla H_{\tau_2}(y_t), \, y_0 = y \right\}.
\]
Then $(\Omega_i)_{i=1}^N$ is an admissible partition in the sense of Definition \ref{d_admissable_single}. We omit the proof of this fact, which can be shown by slightly modifying the proof of \cite[Lemma 3.12]{MS14}. The preceding discussion shows $\nu_i^{\tau}$ defined on $\Omega_i$ by \eqref{e_spGm_local} satisfies PI$(\rho)$ with $\rho^{-1} = O(\tau)$ for all $\tau \leq \tau_2$. \medskip

Equipped with the Poincar\'e inequality for $\nu_i^{\tau}$, the log-Sobolev inequality for $\nu_i^{\tau}$ is now a simple consequence of the following criterion from \cite{MS14}.

\begin{lemma}[Lyapunov condition for local LSI, Theorem 3.15 in \cite{MS14}] Assume that
	\begin{enumerate}
		\item[(i)] There exists a smooth function $W_\tau: \Omega_i \rightarrow (0, \infty)$ and constants $\lambda, b > 0$ such that for $L_\tau:= \tau \Delta - \nabla H \cdot \nabla$
		\begin{equation*} \frac{L_\tau W_\tau}{W_\tau} \leq -\lambda |x|^2 + b \quad \forall x \in \Omega_i.
		\end{equation*}
		\item[(ii)] $\nabla^2 H \geq -K_H$ for some $K_H > 0$ and $\nu_i^\tau$ satisfies PI$(\rho)$.
		\item[(iii)] $W_\tau$ satisfies Neumann boundary condition on $\Omega_i$ (see \eqref{e_ibp}).
	\end{enumerate}
	Then $\nu_i^\tau$ satisfies LSI$(\alpha)$ with
	\begin{equation*}
	\alpha^{-1} \leq 2\sqrt{\frac{\tau}{\lambda} \left(\frac{1}{2} + \frac{b + \lambda \nu_i^\tau(|x|^2)}{\rho \tau}\right)} + \frac{K_H}{\lambda} \left(\frac{1}{2} + \frac{b + \lambda \nu_i^\tau(|x|^2)}{\rho \tau}\right) + \frac{2}{\rho},
	\end{equation*}
	where $\nu_i^\tau(|x|^2)$ denotes the second moment of $\nu_i^\tau$.
\end{lemma}

Choosing $W_\tau$ to be the same Lyapunov function we chose for the PI, it is straightforward to check that, under Assumption \ref{assumeenvLSI}, the conditions (i)-(iii) holds and the second moment $\nu_i^\tau(|x|^2)$ is uniformly bounded. We omit the proofs, which are virtually identical to their counterparts in \cite{MS14} (see Lemmas 3.17-3.19). Finally, $\rho^{-1} = O(\tau)$ yields $\alpha^{-1} = O(1)$. 

\subsection{Proof of Lemma \ref{p_meandiff_prod}}
\label{sc:34}

In order to prove Lemma \ref{p_meandiff_prod}, we observe that the local Gibbs measures $\nu_i^\tau$ are close to a class of truncated Gaussian measures in the sense of mean-difference, see \cite[Lemma 4.6]{MS14}. 

\begin{definition}[Truncated Gaussian measure]\label{d_truncated_Gaussian} Given $m \in \mathbb{R}^n$, $\Sigma$ a symmetric positive definite $n\times n$ matrix, $R \geq 1$, consider the ellipsoid
	\[
	E_i^\tau := \{x\in \mathbb{R}^n: (x-m)\cdot \Sigma^{-1} (x-m) \leq R^2 \tau \}.
	\]
	The truncated Gaussian measure $\gamma^{\tau}$ at temperature $\tau$ with mean $m$ and covariance $\Sigma$ on scale $R$ is defined to be
	\begin{align}
	\gamma^{\tau}(x) :=  \frac{\exp{\left(-\frac{1}{2\tau} (x-m)\cdot \Sigma^{-1} (x-m) \right)}}{Z_R\sqrt{\tau}^n \sqrt{\det \Sigma}} \mathbbm{1}_{E^\tau},
	\end{align}
	where $Z_R := \int_{B_R(0)} \exp{(-|x|^2/2)} dx = \sqrt{2\pi}^n (1 - O(e^{-R^2}R^{n-2}))$.
\end{definition}

\begin{lemma}[Approximation by truncated Gaussian]\label{p_meandiff_approx} For $\tau \leq \tau_2$, let $\gamma_i^\tau$ be the truncated Gaussian measure at temperature $\tau$ with mean $m_i$ and covariance $\Sigma_i = (\nabla H^2 (m_i))^{-1}$ on scale $R(\tau_2) = |\log \tau_2|^{1/2}$. Then
	\begin{equation}\label{e_pdf_approx}
	\frac{d\gamma_i^\tau}{d\nu_i^\tau}(x) = 1 + \omega(\tau_2),
	\end{equation}
	uniformly in the support of $\gamma_i^\tau$, and for any smooth function $f: \mathbb{R}^n \rightarrow \mathbb{R}$
	\begin{align}\label{e_meandiff_approx}  
	(\E_{\nu_i^\tau} f - \E_{\gamma_i^\tau} f)^2 \leq \var_{\nu_i^\tau} \left(\frac{d\gamma_i^\tau}{d\nu_i^\tau}\right) \var_{\nu_i^\tau}(f)\leq \omega(\tau_2) \tau \E_{\nu_i^\tau} |\nabla f|^2.
	\end{align}
	where $\omega(\tau_2):= O(\sqrt{\tau_2} |\log \tau_2|^{3/2})$.
\end{lemma}
We omit the proof of Lemma \ref{p_meandiff_approx}, which is the same as \cite[Lemma 4.6]{MS14} with only minor changes. 
\begin{corollary}\label{p_meandiff_cor1} For any smooth function $f: \mathbb{R}^n \times \mathbb{R}^n \rightarrow \mathbb{R}$
	\begin{align}
	\bigl(\E_{\pi^\sigma_{ij}} f- \E_{\gamma_i^{\tau_{\sigma(1)}} \otimes \gamma_j^{\tau_{\sigma(2)}}} f \bigr)^2 \leq \omega(\tau_2) \E_{\pi^\sigma_{ij}}\bigl(\tau_{\sigma(1)}  |\nabla_{x_1} f|^2 + \tau_{\sigma(2)}  |\nabla_{x_2} f|^2\bigr). 
	\end{align}
\end{corollary}
\begin{proof} This follows from the previous lemma by writing
	\begin{align}
\E_{\pi^\sigma_{ij}} f- \E_{\gamma_i^{\tau_{\sigma(1)}} \otimes \gamma_j^{\tau_{\sigma(2)}}} f &= \bigl(\E_{\nu_i^{\tau_{\sigma(1)}} \otimes \nu_j^{\tau_{\sigma(2)}}} f - \E_{\gamma_i^{\tau_{\sigma(1)}} \otimes \nu_j^{\tau_{\sigma(2)}}} f\bigr) \\ 
&\quad + \bigl(\E_{\gamma_i^{\tau_{\sigma(1)}} \otimes \nu_j^{\tau_{\sigma(2)}}} f - \E_{\gamma_i^{\tau_{\sigma(1)}} \otimes \gamma_j^{\tau_{\sigma(2)}}} f \bigr).
\end{align}
\end{proof}
This reduces our task to proving mean-difference estimate for truncated Gaussian.
\begin{lemma}[Mean-difference estimate for truncated Gaussians at two temperatures]\label{p_meandiff_truncated_Gaussian} For any smooth function $f: \mathbb{R}^n \rightarrow \mathbb{R}$
	\begin{align}
	(\E_{\gamma_i^{\tau_2}} f - \E_{\gamma_i^{\tau_1}} f)^2 &\leq 
	C_n \|\Sigma_i\| \left(1 +\Phi_n\Bigl(\frac{\tau_2}{\tau_1}\Bigr) \right) \tau_2 \E_{\gamma_i^{\tau_2}} |\nabla f|^2,
	 \label{e_meandiff_truncated_Gaussian}
	\end{align}
	where the function $\Phi_n$ is given by \eqref{eq:temperature_dependence}, and $C_n$ is a constant only depending on $n$. 
\end{lemma}

\begin{proof} By change of variables, it suffices to show the first inequality for $m_i = 0, \Sigma_i = \Id$. From the Cauchy-Schwarz inequality and the fundamental theorem of calculus, we can deduce
	\begin{align*}
	(\E_{\gamma_i^{\tau_2}} f - \E_{\gamma_i^{\tau_1}} f)^2 & \leq \E_{\gamma_i^{1}}\bigl(f(\sqrt{\tau_2} X) - f(\sqrt{\tau_1} X)\bigr)^2 \\
	&\leq \int_{S^{n-1}} d\omega \int_0^R \left(\int_{\sqrt{\tau_1} r }^{\sqrt{\tau_2} r} |\nabla f(s \omega)| ds \right)^2 \frac{e^{-\frac{r^2}{2}}}{Z_R}  r^{n-1} dr\\
	&\leq 2 (I_1 + I_2),
	\end{align*}
	where, we recall that $R\geq 1$ from Definition~\ref{d_truncated_Gaussian}
	\begin{align*}
	I_1 &:= \int_{S^{n-1}} d\omega \int_0^R \left(\int_{\sqrt{\tau_1} r }^{\sqrt{\tau_2} r} |\nabla f(s \omega)|\mathbbm{1}_{s\leq  \sqrt{\tau_2}} ds \right)^2 \frac{e^{-\frac{r^2}{2}}}{Z_R}  r^{n-1} dr, \\
	I_2 &:=  \int_{S^{n-1}} d\omega \int_0^R \left(\int_{\sqrt{\tau_1} r }^{\sqrt{\tau_2} r} |\nabla f(s \omega)|\mathbbm{1}_{s>  \sqrt{\tau_2}} ds \right)^2 \frac{e^{-\frac{r^2}{2}}}{Z_R}  r^{n-1} dr.
	\end{align*}	
	Estimate for $I_2$: By Cauchy-Schwarz,
	\begin{align*}
	I_2 &\leq  \int_{S^{n-1}} d\omega \int_0^R (\sqrt{\tau_2} r - \sqrt{\tau_1} r) \left(\int_{\sqrt{\tau_2}}^{R \sqrt{\tau_2} } |\nabla f(s \omega)|^2 \mathbbm{1}_{s\leq r \sqrt{\tau_2}} ds\right) \frac{e^{-\frac{r^2}{2}}}{Z_R}  r^{n-1} dr \\
	&\leq \sqrt{\tau_2} \int_{S^{n-1}} d\omega \int_{\sqrt{\tau_2}}^{R \sqrt{\tau_2}} |\nabla f(s \omega)|^2 \left(\int_{\frac{s}{\sqrt{\tau_2}}}^R   \frac{e^{-\frac{r^2}{2}}}{Z_R}  r^n dr \right) ds.
	\end{align*}
	Using integration by parts and standard Gaussian tail bound, for $s\geq \sqrt{\tau_2}$,
	\[
	\int_{\frac{s}{\sqrt{\tau_2}}}^R   e^{-\frac{r^2}{2}}  r^n dr \leq C_n e^{-\frac{s^2}{2\tau_2}} \left(\frac{s^2}{\tau_2} \right)^{\frac{n-1}{2}}  
	\]
	where $C_n$ is a constant only depending on $n$. This gives 
	\begin{align*}
	I_2 \leq C_n \tau_2 \E_{\gamma_i^{\tau_2}} |\nabla f|^2. 
	\end{align*}
	Estimate for $I_1$: By Cauchy-Schwarz 
	\begin{align*} I_1 &\leq \int_{S^{n-1}} d\omega \int_0^R \left(\int_{0 }^{\sqrt{\tau_2} } |\nabla f(s \omega)|^2 s^{n-1} ds \right) \left(\int_{\sqrt{\tau_1} r }^{\sqrt{\tau_2} r  } s^{-(n-1)} ds \right) \frac{e^{-\frac{r^2}{2}}}{Z_R}  r^{n-1} dr \\
	&= \frac{1}{Z_R} \| \nabla f \|^2_{L^2(B_{\sqrt{\tau_2}}(0))}  \int_0^R \left(\int_{\sqrt{\tau_1} }^{\sqrt{\tau_2}   } u^{-(n-1)} du \right) r e^{-\frac{r^2}{2}} dr \\
	&\leq C_n e^{\frac{1}{2}} \tau_2  \E_{\gamma_i^{\tau_2}} |\nabla f|^2 \cdot \Phi_n\Bigl(\frac{\tau_2}{\tau_1}\Bigr),
	\end{align*}
	where $C_n$ is a constant only depending on $n$.
\end{proof}
\begin{corollary}\label{p_meandiff_cor2} For any smooth function $f: \mathbb{R}^n \times \mathbb{R}^n \rightarrow \mathbb{R}$
	\begin{align*}
	(\E_{\gamma_i^{\tau_{1}} \otimes \gamma_j^{\tau_{2}}} f - \E_{\gamma_i^{\tau_{2}} \otimes \gamma_j^{\tau_{1}}} f)^2 \leq \left(1 +\Phi_n\Bigl(\frac{\tau_2}{\tau_1}\Bigr) \right) O(\tau_2) \bigl(\E_{\pi^+_{ij}} |\nabla_{x_2} f|^2 + \E_{\pi^-_{ij}} |\nabla_{x_1} f|^2 \bigr).  
	\end{align*}
\end{corollary}

\begin{proof} This follows from the previous lemma and \eqref{e_pdf_approx} by writing
	\begin{align*}
	\E_{\gamma_i^{\tau_{1}} \otimes \gamma_j^{\tau_{2}}} f - \E_{\gamma_i^{\tau_{2}} \otimes \gamma_j^{\tau_{1}}} f = (\E_{\gamma_i^{\tau_{1}} \otimes \gamma_j^{\tau_{2}}} f - \E_{\gamma_i^{\tau_{1}} \otimes \gamma_j^{\tau_{1}}} f ) + (\E_{\gamma_i^{\tau_{1}} \otimes \gamma_j^{\tau_{1}}} f - \E_{\gamma_i^{\tau_{2}} \otimes \gamma_j^{\tau_{1}}} f).
	\end{align*}
\end{proof}
Lemma \ref{p_meandiff_prod} follows from Corollary \ref{p_meandiff_cor1} and \ref{p_meandiff_cor2}.
\begin{remark} One can show a weaker version of Lemma \ref{p_meandiff_prod} by a simpler approach: First we split the mean-difference as
	\begin{align}
	(\E_{\pi^+} f - \E_{\pi^-} f)^2 \leq 2\E_{\nu_i^{\tau_1}}(\E_{\nu_j^{\tau_2}} f - \E_{\nu_j^{\tau_1}} f)^2 + 2\E_{\nu_j^{\tau_1}}(\E_{\nu_i^{\tau_1}} f - \E_{\nu_i^{\tau_2}} f)^2 	
	\end{align} 
	Now, using the covariance representation of mean-difference and Cauchy-Schwarz
	\begin{align}
	(\E_{\nu_j^{\tau_2}} f - \E_{\nu_j^{\tau_1}} f)^2 \leq \var_{\nu_j^{\tau_2}} (f) \var_{\nu_j^{\tau_2}}\biggl(\frac{d\nu_j^{\tau_1}}{d\nu_j^{\tau_2}}\biggr) \leq O(\tau_2) \E_{\nu_j^{\tau_2}} |\nabla_{x_2} f|^2 \E_{\nu_j^{\tau_1}}\biggl(\frac{d\nu_j^{\tau_1}}{d\nu_j^{\tau_2}}\biggr).
	\end{align}
	Finally, using the partition size given in \eqref{e_spGM_admissible} we have
	\begin{align}
	\frac{d\nu_j^{\tau_1}}{d\nu_j^{\tau_2}} = \frac{\nu_j^{\tau_2}(\Omega_j)}{\nu_j^{\tau_1}(\Omega_j)} e^{-H(x)(\tau_1^{-1} - \tau_2^{-1})} \leq \frac{\nu_j^{\tau_2}(\Omega_j)}{\nu_j^{\tau_1}(\Omega_j)} \leq \sqrt{\tau_2/\tau_1}^n (1+O(\sqrt{\tau_2}|\ln \tau_2|^{3/2})).
	\end{align}
\end{remark}

\subsection{Proof of Proposition \ref{p_opt_local}}\label{s_opt_local_PI} 

It suffices to consider test functions of the form $f(x, y) = f(x)$. This is equivalent to replacing $\mu$ by its first marginal, which is $\bar\mu = \frac{1}{2}(\nu^{\tau_1} + \nu^{\tau_2})$. In this case,  $\var_\mu (f)$ and $\mathcal{E}_\mu (f)$ reduces to
\begin{align}
\var_{\bar{\mu}}(f) &= \frac{1}{2}( \var_{\nu^{\tau_1}}(f) + \var_{\nu^{\tau_2}} (f)) + \frac{1}{4} (\E_{\nu^{\tau_1}} f - \E_{\nu^{\tau_2}} f)^2, \\
\mathcal{E}_{\bar{\mu}}(f) &= \frac{1}{2}(\tau_1 \E_{\nu^{\tau_1}} |\nabla f|^2 + \tau_2 \E_{\nu^{\tau_2}} |\nabla f|^2).
\end{align}
We further restrict $f$ to $C_c (\Omega_1)$. Recall the notation $\approx, \lessapprox$ defined in \eqref{eq:approx_notation}. By \eqref{e_spGM_admissible} and \eqref{eq:nondeg}, $\nu^{\tau_1}(\Omega_1), \nu^{\tau_2}(\Omega_1) \approx 1$ once $\tau_1, \tau_2$ are small enough, so $\frac{d\nu^{\tau_1}_1}{d\nu^{\tau_1}}, \frac{d\nu^{\tau_2}_1}{d\nu^{\tau_2}} \approx 1$ on $\Omega_1$ (see equation \eqref{e_spGm_local}). A crude application of Young's inequality then yields
\begin{align*}
\var_{\bar \mu}(f) &\gtrsim (\E_{\nu^{\tau_1}} f)^2 - 4(\E_{\nu^{\tau_2}} f)^2 \gtrsim
(\E_{\nu^{\tau_1}_1} f)^2 - 5(\E_{\nu^{\tau_2}_1} f)^2, \\
\mathcal{E}_{\bar \mu}(f)  &\lesssim
\tau_1 \E_{\nu^{\tau_1}_1} |\nabla f|^2 + \tau_2 \E_{\nu^{\tau_2}_1} |\nabla f|^2,
\end{align*}
where $\lesssim$ means $\leq$ up to a multiplicative constant. By change of variables, we may assume $m_1 = 0, \Sigma_1 = (\nabla^2 H(m_1))^{-1} = \Id$. We consider a test function of the form 
\[f(x) = f_\varepsilon (x) = h(|x|/\sqrt{\varepsilon}),\]
where $h \geq 0$ is a compactly supported, absolutely continuous function and $\tau_1 \leq \varepsilon \leq \tau_2$ is a scaling parameter, both to be specified later. As in the proof of Lemma \ref{p_meandiff_prod}, we will approximate by truncated Gaussian measures (see Definition \ref{d_truncated_Gaussian}). Since $\varepsilon \leq \tau_2$, $f_\varepsilon$ is supported in the support of $\gamma_1^{\tau_2}$. By Lemma \ref{p_meandiff_approx}, 
\begin{align}
\var_{\bar \mu}(f) &\gtrsim  (\E_{\gamma^{\tau_1}_{1}} f_\varepsilon)^2 - 6(\E_{\gamma^{\tau_2}_{1}} f_\varepsilon)^2
\label{e_opt_local_var}, \\
\mathcal{E}_{\bar \mu}(f)  &\lesssim \tau_1 \E_{\nu^{\tau_1}_1} |\nabla f_\varepsilon|^2 + \tau_2 \E_{\gamma^{\tau_2}_1} |\nabla f_\varepsilon|^2 \label{e_opt_local_dirichlet},
\end{align}
if $\tau_2$ is small enough. By rescaling, we have:
\begin{align} 
\tau_1 \E_{\nu^{\tau_1}_1} |\nabla f_\varepsilon|^2 &= \frac{\tau_1}{\varepsilon} \E_{\nu^{\frac{\tau_1}{\varepsilon}}_1} |\nabla f_1|^2  \label{e_opt_local_1}, \\
\tau_2 \E_{\gamma^{\tau_2}_1} |\nabla f_\varepsilon|^2 &= \frac{\tau_2}{\varepsilon} \E_{\gamma^{\frac{\tau_2}{\varepsilon}}_1} |\nabla f_1|^2 \leq  \frac{1}{\sqrt{2\pi}^n} (\varepsilon/\tau_2)^{(n-2)/2} \|\nabla f_1\|^2_{L^2} \label{e_opt_local_2}, \\
\E_{\gamma^{\tau_2}_{1}} f_\varepsilon &= \E_{\gamma^{\frac{\tau_2}{\varepsilon}}_{1}} f_1 \leq \frac{1}{\sqrt{2\pi}^n} (\varepsilon/\tau_2)^{n/2} \|f_1\|_{L^1} \label{e_opt_local_3},
\end{align}
and for any $r\geq 0$,
\begin{align}
\E_{\gamma^{\tau_1}_{1}} f_\varepsilon = \E_{\gamma^{\frac{\tau_1}{\varepsilon}}_{1}} f_1 \geq  P_{\gamma^{\frac{\tau_1}{\varepsilon}}_1}(|X| \leq r) \cdot \inf_{|x| \leq r} f_1  \geq  \left(1 - n e^{-\frac{r^2}{2n} \frac{\varepsilon}{\tau_1}} \right) \cdot \inf_{[0, r]} h \label{e_opt_local_4}.
\end{align}

In the following $R_n > 0$ is the number such that $\exp{\left(-\frac{R_n^2}{2n} \right)} = \frac{1}{2}$.

Case 1: $n\geq 3$. We choose $h$ to be a compactly supported smooth function such that $h = 1$ on $[0, R_n]$, decreases to $0$ on $[R_n, 2R_n]$ and is $0$ outside $[0, 2R_n]$. Then
\begin{equation*} 
\tau_2 \E_{\gamma^{\tau_2}_1} |\nabla f_\varepsilon|^2 \overset{\eqref{e_opt_local_2}}{\lesssim} (\varepsilon/\tau_2)^{(n-2)/2}, \,\, \E_{\gamma^{\tau_2}_{1}} f_\varepsilon \overset{\eqref{e_opt_local_3}}{\lesssim}  (\varepsilon/\tau_2)^{n/2}, \,\, \E_{\gamma^{\tau_1}_{1}} f_\varepsilon \overset{\eqref{e_opt_local_4}}{\geq} \frac{1}{2},
\end{equation*}
where the implicit constants only depend on the dimension $n$ and the function $h$. Since $h' = 0$ on $[0, R_n]$
\begin{align*}
\tau_1 \E_{\nu^{\tau_1}_1} |\nabla f_\varepsilon|^2 \overset{\eqref{e_opt_local_1}}{\leq} \frac{\tau_1}{\varepsilon}
\|h'\|^2_{L^\infty} P_{\nu^{\frac{\tau_1}{\varepsilon}}_1} (|X|\geq R_n)  \leq \frac{\tau_1}{\varepsilon}
\|h'\|^2_{L^\infty} C_H e^{-c_H \frac{\varepsilon}{\tau_1}} \lesssim_m (\tau_1/\varepsilon)^{m},
\end{align*}
for every positive integer $m$, where the constants $c_H, C_H > 0$ only depend on the Hamiltonian $H$. The second inequality is a consequence of Assumption \ref{assumeenv} (see \cite[Lemma 3.13]{MS14}). Now, for any $0 < \eta < \frac{1}{2}$, set $\varepsilon = \tau_1^{1-\eta} \tau_2^\eta$, and choose $m$ large enough so that $\eta m \geq (1-\eta)(n-2)/2$, we obtain
\begin{align*}
\mathcal{E}_{\bar \mu}(f)  \overset{\eqref{e_opt_local_dirichlet}}{\lesssim_{\eta}} (\tau_1/\tau_2)^{(1-\eta)(n-2)/2}, \,\,
\var_{\bar \mu}(f) &\overset{\eqref{e_opt_local_var}}{\gtrsim_{\eta}} (\tau_2/\tau_1)^{(1-\eta)(n-2)/2} \mathcal{E}_{\bar \mu}(f),
\end{align*}
if $\tau_2$, $\tau_1/\tau_2$ are both small enough. \medskip

Case 2: $n=2$. Let $h$ be the function given by
\begin{equation*}
h(r) = \begin{cases} 1 &  \mbox{ for } 0\leq r\leq r_0 \\
2 (1-r^\alpha) &  \mbox{ for } r_0\leq r\leq 1 \\
0 &  \mbox{ for } r\geq 1,
\end{cases}
\end{equation*} 
for parameters $0 < \alpha < 1, 0 < r_0 < 1$ satisfying $r_0^\alpha = \frac{1}{2}$, to be specified later. Then $h$ is absolutely continuous, $h'=0$ on $[0, r_0]$, and by direct computation
\begin{align*}
\|f_1\|_{L^1} \leq \pi\alpha, \quad \|\nabla f_1\|^2_{L^\infty} = \alpha^2 r_0^{-2}, \quad \|\nabla f_1\|_{L^2}^2 = 3\pi \alpha.
\end{align*}
We choose $\varepsilon = \tau_2$ and $r_0^2 \frac{\tau_2}{\tau_1} = R_2^2$ (which is possible once $\tau_1/\tau_2$ is small enough). Then:
\begin{align*}
\E_{\gamma^{\tau_2}_{1}} f_\varepsilon &\overset{\eqref{e_opt_local_3}}{\leq}  \frac{1}{2\pi} \frac{\varepsilon}{\tau_2} \|f_1\|_{L^1} \leq  \frac{\alpha}{2}, \,\,
\E_{\gamma^{\tau_1}_{1}} f_\varepsilon \overset{\eqref{e_opt_local_4}}{\geq} \frac{1}{2},   \\
\tau_1 \E_{\nu^{\tau_1}_1} |\nabla f_\varepsilon|^2 &\overset{\eqref{e_opt_local_1}}{\leq} \frac{\tau_1}{\varepsilon} \|\nabla f_1\|^2_{L^\infty} \leq \frac{\alpha^2}{R_2^2}, \,\,
\tau_2 \E_{\gamma^{\tau_2}_1} |\nabla f_\varepsilon|^2 \overset{\eqref{e_opt_local_2}}{\leq}  \frac{1}{2\pi}  \|\nabla f_1\|^2_{L^2} = \frac{3\alpha}{2} .
\end{align*}
Since $r_0^\alpha = \frac{1}{2}$,
$\frac{1}{\alpha} = \frac{1}{2\ln 2}\ln \left(\frac{\tau_2}{\tau_1 R_2^2}\right)$. Thus
\begin{align*}
\mathcal{E}_{\bar \mu}(f)  \overset{\eqref{e_opt_local_dirichlet}}{\lesssim} \frac{\alpha^2}{R_2^2} + \frac{3\alpha}{2}, \quad
\var_{\bar \mu}(f) \overset{\eqref{e_opt_local_var}}{\gtrsim} \frac{1}{\alpha} \mathcal{E}_{\bar \mu}(f) \gtrsim \ln \left(\frac{\tau_2}{\tau_1}\right) \mathcal{E}_{\bar \mu}(f),
\end{align*}
if $\tau_2, \tau_1/\tau_2$ are both small enough.

\subsection{Proof of Proposition \ref{p_opt_EK_1d} and Proposition \ref{p_opt_EK_1d_LSI}}\label{s_proof_opt_EK_1d} It suffices to consider test functions of the form $f(x, y) = g(x)g(y)$. This is equivalent to replacing $\mu$ by $\pi = \nu^{\tau_1} \otimes \nu^{\tau_2}$. In this case, $\var_\mu (f), \ent_\mu(f^2), \mathcal{E}_\mu (f), \mathcal{I}_\mu (f)$ reduce to
\begin{align}
\var_\pi (f) &= \E_{\nu^{\tau_1}} g^2 \E_{\nu^{\tau_2}} g^2 - (\E_{\nu^{\tau_1}} g)^2 (\E_{\nu^{\tau_2}} g)^2, \\
\ent_\pi (f) &= \E_{\nu^{\tau_1}} g^2 \ent_{\nu^{\tau_2}} g^2 + \E_{\nu^{\tau_2}} g^2 \ent_{\nu^{\tau_1}} g^2, \\
\frac{1}{2}\mathcal{I}_\pi (f^2) = \mathcal{E}_\pi (f) &= \tau_1  \E_{\nu^{\tau_1}} (g')^2 \E_{\nu^{\tau_2}} g^2 + \tau_2  \E_{\nu^{\tau_1}} g^2 \E_{\nu^{\tau_2}} (g')^2.
\end{align}

We represent $\nu^{\tau_i}$ for $i=1,2$ as the mixture
\[
\nu^{\tau_i} = Z_1^{\tau_i} \nu_1^{\tau_i} + Z_2^{\tau_i} \nu_2^{\tau_i} \quad \text{where } \nu_1^{\tau_i}:= \nu^{\tau_i}|_{\Omega_1}, \nu_2^{\tau_i}:= \nu^{\tau_i}|_{\Omega_2},
\]
where $\Omega_1 := (-\infty, s), \Omega_2: = (s, \infty)$. Recall the notation $\approx, \lessapprox$ defined in \eqref{eq:approx_notation}. Denote
\begin{align}
Z_1^{\tau_i} = \nu^{\tau_i} (\Omega_1)  \approx 1, \quad
Z_2^{\tau_i}  = \nu^{\tau_i} (\Omega_2) \approx \frac{\sqrt{H''(m_1)}}{\sqrt{H''(m_2)}} e^{-H(m_2)/\tau_i}.
\end{align}

Proof of Proposition \ref{p_opt_EK_1d} (\textbf{Optimality of PI in 1d}): Imposing $\E_{\nu^{\tau_1}} g = 0$, we get
\[
\frac{\mathcal{E}_\pi (f)}{\var_\pi (f)} = \tau_1 \frac{\E_{\nu^{\tau_1}} (g')^2}{\E_{\nu^{\tau_1}} g^2} + \tau_2 \frac{\E_{\nu^{\tau_2}} (g')^2}{\E_{\nu^{\tau_2}} g^2}.
\]

We make the following ansatz for $g$:
\begin{equation}
g(x) = \begin{cases} g(m_1) &  \mbox{ for } x \leq s-\delta \\
g(m_1) + \frac{g(m_2) - g(m_1)}{\sqrt{2\pi \sigma \tau_2}} \cdot \kappa \int_{s-\delta}^x e^{-(y-s)^2/(2\sigma \tau_2)} dy &  \mbox{ for } s-\delta < x < s + \delta \\
g(m_2) & \mbox{ for } x > s+\delta,
\end{cases}
\end{equation}
where $\sigma$ is a positive constant to be specified later, $\delta = \sqrt{2r_0 \tau_2 |\ln \tau_2|}$ for some positive constant $r_0$ to be chosen later, and $\kappa$ is chosen so that $g$ is continuous at $s+\delta$. (This is the same kind of ansatz used in \cite[Section 2.4]{MS14}.) Then $\kappa = 1 + O(\tau_2^{-r_0/\sigma}) \approx 1$ once $r_0$ is large enough. Fix such a choice of $r_0$. For $\tau_2$ small enough, $\delta$ is small enough so that
\begin{align}
\E_{\nu^{\tau_i}} g \approx g(m_1)Z_1^{\tau_i} + g(m_2)Z_2^{\tau_i}.
\end{align}
This motivates the choice
\[g(m_1) \approx -1, g(m_2) \approx 1/Z_2^{\tau_1},
\]
such that $\E_{\nu^{\tau_1}} g = 0$. Then
\begin{align}
\E_{\nu^{\tau_2}} g^2 &\approx Z_1^{\tau_2} g(m_1)^2 + Z_2^{\tau_2} g(m_2)^2 \approx g(m_2)^2 Z_2^{\tau_2}, \\
\E_{\nu^{\tau_1}} g^2 &\approx Z_1^{\tau_1} g(m_1)^2 + Z_2^{\tau_1} g(m_2)^2 \approx g(m_2)^2 Z_2^{\tau_1} .
\end{align}
Finally, we compute the Dirichlet forms. By Taylor expansion of $H$ around $s$
\begin{align}
\E_{\nu^{\tau_2}} (g')^2 &\approx \frac{g(m_2)^2}{2\pi\sigma \tau_2} \frac{1}{Z^{\tau_2}} \int_{B_\delta(s)} e^{-(x-s)^2/(\sigma \tau_2) - H(x)/\tau_2} dx \\
&\approx \frac{g(m_2)^2}{2\pi \sigma \tau_2} \frac{\sqrt{H''(m_1)}}{\sqrt{2\pi \tau_2}} e^{-H(s)/\tau_2} \int_{B_\delta(s)} e^{-(x-s)^2/(2\tau_2)(2/\sigma + H''(s))} dx\\
&\approx g(m_2)^2 \frac{\sqrt{H''(m_1)}}{2\pi \tau_2} e^{-H(s)/\tau_2} \sqrt{|H''(s)|},
\end{align} 
where we set $\sigma = 1/|H''(s)| = -1/H''(s)$. This implies
\[
\tau_2 \frac{\E_{\nu^{\tau_2}} (g')^2}{\E_{\nu^{\tau_2}} g^2} \approx \frac{\sqrt{H''(m_2)|H''(s)|}}{2\pi} e^{(H(m_2)-H(s))/\tau_2} \approx \rho.
\]
It remains to show the other term is asymptotically negligible:
\begin{align}
\E_{\nu_{\tau_1}} (g')^2 &\lessapprox \frac{g(m_2)^2}{2\pi \sigma \tau_2} \frac{1}{Z_{\tau_1}} \int_{B_\delta(s)} e^{-(x-s)^2/(\sigma \tau_2)} dx \cdot \sup_{x\in B_\delta(s)} e^{-H(x)/\tau_1} \\
&\lessapprox \frac{g(m_2)^2}{2\pi} \frac{\sqrt{H''(m_1) |H''(s)|}}{ \sqrt{2\tau_1\tau_2 } } e^{-(1-\eta)H(s)/\tau_1},
\end{align}
where $\eta = O(\delta^2)$. Since $\tau_2 > K\tau_1$ for a constant $K > 1$, choosing $\delta$ sufficiently small, this implies $\tau_1 \frac{\E_{\nu^{\tau_1}} (g')^2}{\E_{\nu^{\tau_1}} g^2}$ is asymptotically negligible compared to $\rho$. 

Proof of Proposition \ref{p_opt_EK_1d_LSI} (\textbf{Optimality of LSI in 1d up to constant factor}): In the same set-up as above, imposing $\E_{\nu^{\tau_1}} g^2 = 1$, we get
\[
\frac{1}{2}\frac{\mathcal{I}_\pi (f^2)}{\ent_\pi (f)} \leq \tau_1 \frac{\E_{\nu^{\tau_1}} (g')^2}{\ent_{\nu^{\tau_1}} g^2 } + \tau_2 \frac{\E_{\nu^{\tau_2}} (g')^2}{\ent_{\nu^{\tau_1}} g^2 \E_{\nu^{\tau_2}} g^2 }.
\]
We use the same form of ansatz as before with
\begin{align}
g(m_1)^2 \approx \frac{Z^{\tau_1}_2}{Z^{\tau_1}_1} \approx \frac{\sqrt{H''(m_1)}}{\sqrt{H''(m_2)}} e^{-H(m_2)/\tau_1}, \quad g(m_2)^2 = \frac{1}{g(m_1)^2}
\end{align}
such that $\E_{\nu^{\tau_1}} g^2 = 1$. Then
\begin{align}
\E_{\nu^{\tau_2}} g^2 &\approx Z_1^{\tau_2} g(m_1)^2 + Z_2^{\tau_2} g(m_2)^2 \approx Z_2^{\tau_2} g(m_2)^2, \\ 
\ent_{\nu^{\tau_1}} g^2 &
\approx Z_1^{\tau_1} g(m_1)^2 \ln g(m_1)^2 + Z_2^{\tau_1} g(m_2)^2 \ln g(m_2)^2 \approx \ln g(m_2)^2 \approx \frac{H(m_2)}{\tau_1} ,
\end{align}
and the same computation as before shows
\begin{align}
\E_{\nu^{\tau_1}} (g')^2 &\lessapprox g(m_2)^2\frac{\sqrt{H''(m_1) |H''(s)|}}{2\pi \sqrt{2\tau_1\tau_2 }} e^{-(1-\eta)H(s)/\tau_1}, \\
\E_{\nu^{\tau_2}} (g')^2 &\approx g(m_2)^2\frac{\sqrt{H''(m_1) |H''(s)|}}{2\pi \tau_2} e^{-H(s)/\tau_2},
\end{align}
where $\eta = O(\delta^2)$. This implies
\begin{align}
\tau_2 \frac{\E_{\nu^{\tau_2}} (g')^2}{\ent_{\nu^{\tau_1}} g^2  \E_{\nu^{\tau_2}} g^2} \approx \tau_1  \frac{\sqrt{H''(m_2)|H''(s)|}}{2\pi H(m_2)} e^{(H(m_2)-H(s))/\tau_2}  \lesssim \alpha,
\end{align}
and that $\tau_1 \frac{\E_{\nu^{\tau_1}} (g')^2}{\ent_{\nu^{\tau_1}} g^2}$ is asymptotically negligible compared to $\alpha$. 

\section*{Acknowledgment}
The authors want to thank Max Fathi and Paul Bressloff for the fruitful discussions. 
GM and AS want to thank the University of Bonn for financial support via the CRC 1060 \emph{The Mathematics of Emergent
Effects} of the University of Bonn that is funded through the Deutsche Forschungsgemeinschaft (DFG, German Research Foundation).
AS also is funded by the DFG under Germany's Excellence Strategy EXC 2044--390685587, Mathematics M\"unster: Dynamics--Geometry--Structure.
WT gratefully acknowledges financial support through an NSF grant DMS-2113779 and a start-up grant at Columbia University.

\setlength\parskip{0pt}

\bibliographystyle{alphaabbr}
\bibliography{bib}

\end{document}